\documentclass{amsart}

\usepackage{amsmath}
\usepackage{graphicx}
\usepackage{amsfonts}
\usepackage{amssymb}
\usepackage[utf8]{inputenc}
\usepackage{fancyhdr}
\usepackage{amsaddr}
\usepackage{fancyvrb}
\usepackage[normalem]{ulem}

\usepackage{datetime}
\usepackage{amsthm}
\usepackage{algorithm, caption}
\usepackage[noend]{algpseudocode}
\usepackage{bbm}
\usepackage{xcolor}
\usepackage[hidelinks]{hyperref}
\usepackage{colonequals}
\usepackage{comment}
\usepackage{listings}
\usepackage{tikz}

\usepackage[all]{xy} \SelectTips{eu}{}

\newtheorem{thm}{Theorem}[section]

  \newtheorem{con}[thm]{Conjecture}
  \newtheorem{lem}[thm]{Lemma}

  \newtheorem{qst}[thm]{Question}

  \theoremstyle{definition}
  \newtheorem{rem}[thm]{Remark}
  \newtheorem{defn}[thm]{Definition}
  \newtheorem{exam}[thm]{Example}
  \newtheorem{num}[thm]{} 
  \newtheorem{code}[thm]{Code}

  \numberwithin{equation}{thm}

  \makeatletter
  \let\c@algorithm\c@thm
  \makeatother

\newcommand{\F}{\mathbbm{F}}
\newcommand{\K}{\mathbbm{k}}

\newcommand{\X}{\mathbb{X}}
\newcommand{\Q}{\mathbb{Q}}
\newcommand{\init}{\operatorname{in}}
\newcommand{\ndiv}{\!\nmid\!}
\newcommand{\nf}{\operatorname{nf}}

\newcommand{\m}{\mathfrak{m}}
\newcommand{\fI}{\mathfrak{I}}

\newcommand{\fa}{\mathfrak{a}}
\newcommand{\fb}{\mathfrak{b}}

\newcommand{\cL}{\mathcal{L}}
\newcommand{\cF}{\mathcal{F}}
\newcommand{\cC}{\mathcal{C}}
\newcommand{\cG}{\mathcal{G}}
\newcommand{\Flags}{\mathcal{LF}}

\newcommand{\tagform}[1]{\textnormal{\small(#1)}}
\newcommand{\tagKFlin}{\tagform{KF1}}
\newcommand{\tagKFmax}{\tagform{KF2}}
\newcommand{\tagKFcol}{\tagform{KF3}}

\newcommand{\tagLF}{\tagform{LF}}

\title{Constructing Koszul filtrations: existence and non-existence for G-quadratic algebras}
\author{Emily Berghofer, Lisa Nicklasson, Peder Thompson,\\ and Thomas Westerbäck}\author{\vspace{-30pt}}
\address{Mälardalen University}
\email{emily.berghofer@mdu.se}
\email{lisa.nicklasson@mdu.se}
\email{peder.thompson@mdu.se}
\email{thomas.westerback@mdu.se}

\date{Tuesday 1\textsuperscript{st} December, 2026}

\begin{document}

\begin{abstract}
Given a standard graded algebra over a field, we consider the
relationship between G-quadraticity and the existence of a Koszul filtration. It is known that an algebra with a Koszul filtration is not necessarily G-quadratic. In this paper we present the first example of an algebra which is G-quadratic but does not have a Koszul filtration, resolving a conjecture of Ene, Herzog, and Hibi. Further, we explore under which conditions Koszul filtrations do exist, and 
we develop both theoretical and computational tools for constructing them.
\end{abstract}
\maketitle

\section{Introduction}
\noindent
For a (commutative) standard graded algebra over a field, linearity of the minimal free resolution of its residue field defines the Koszul property. Introduced by Priddy \cite{Priddy1970}, this property can be challenging to algorithmically verify, as first exhibited by Roos \cite{Roos1993}. Indeed, McCullough and Seceleanu \cite{MS20} showed that, even for quadratic Gorenstein algebras, no fixed finite number of linear steps in the resolution of the residue field suffices to test for Koszulness. Consequently, it is of interest to develop sufficient conditions for Koszulness. While the existence of a quadratic Gröbner basis is a classic method to identify a Koszul algebra, Koszul filtrations were introduced by Conca, Trung, and Valla \cite{CTV2001} as another powerful tool for detecting the Koszul property, used for example by Conca, Rossi, and Valla in \cite{CRV2001}. More recent work of Mastroeni and McCullough \cite{MM2023} and Rice \cite{Rice2023} has further highlighted the utility of Koszul filtrations. In this paper, we examine the relationship between the existence of quadratic Gröbner bases and that of Koszul filtrations, with an emphasis on the construction of Koszul filtrations.

These conditions are related via the following hierarchy. The existence of a Gröbner flag implies that an algebra is both G-quadratic---that is, it has a quadratic Gröbner basis after a linear change of variables---and has a Koszul filtration (Section \ref{sec_prelim} explains these notions), and these conditions imply Koszulness; see \cite{CRV2001}. However, all of these implications are generally strict \cite{CRV2001}. It is also known that an algebra with a Koszul filtration is not necessarily G-quadratic. Explicit examples are given by Conca \cite{Conca2014} and D'Al\`i \cite{DAli2017,DAli2025}, and we give another such algebra in Example \ref{ex_Roos}. Other important conditions that imply Koszulness include universally Koszul, LG-quadratic, and absolutely Koszul; see e.g. Conca, De Negri, and Rossi \cite{conca2012koszulalgebrasregularity}. The diagram below summarizes the relations we focus on in this paper, where the dotted arrows indicate non-implications.
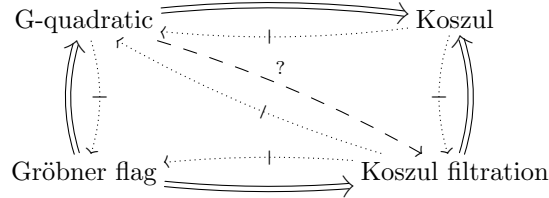
\[\xymatrixrowsep{15mm}\xymatrixcolsep{25mm}
\xymatrix{
\text{Gröbner flag} \ar@/^0.5pc/@{=>}[r] \ar@/_0.5pc/@{=>}[d]  & \text{G-quadratic} \ar@/^0.5pc/@{=>}[d]\ar@/^0.5pc/@{.>}[l] |{\SelectTips{cm}{}\object@{|}}|{} \ar@/_0.5pc/@{-->}[dl]_{?}\\
\text{Koszul filtration} \ar@/_0.5pc/@{=>}[r] \ar@/_0.5pc/@{.>}[u] |{\SelectTips{cm}{}\object@{|}}|{} \ar@/_0.5pc/@{.>}[ur] |{\SelectTips{cm}{}\object@{|}}|{}& \text{Koszul} \ar@/^0.5pc/@{.>}[u] |{\SelectTips{cm}{}\object@{|}}|{}  \ar@/_0.5pc/@{.>}[l] |{\SelectTips{cm}{}\object@{|}}|{}
}\]

\noindent
So far, the status of the diagonal dashed implication arrow has been unknown:
\begin{qst}
Are there G-quadratic algebras without Koszul filtrations?
\end{qst}
The relationship between G-quadraticity and the existence of a Koszul filtration was investigated by Ene, Herzog, and Hibi \cite{Ene_2015}, who conjectured that the property of being G-quadratic need not imply the existence of a Koszul filtration. One of our main results, Theorem \ref{thm_b4t}, resolves this conjecture by providing the first such example:
\begin{thm}\label{thm_A}
There exists a G-quadratic algebra with no Koszul filtration.
\end{thm}
\noindent
The example is motivated by recent work of LaClair, Mastroeni, McCullough, and Peeva \cite{LMMP_2025} on Koszulness of graded Möbius algebras. Our proof utilizes computations in Macaulay2, as well as algorithms for computing Koszul filtrations that we develop in Section \ref{sec_alg}. It also shows that the existence of a Koszul filtration depends on the characteristic of the field.

For many nice classes of algebras being G-quadratic does, however, imply the existence of a Koszul filtration. 
For example, quadratic monomial algebras (see \cite{conca2012koszulalgebrasregularity}) and Hibi rings (see Lu and Zhang \cite{LZ17}) satisfy this property. Recent work of D'Al\`i considers another nice class of G-quadratic algebras which are strongly Koszul, and thus have Koszul filtrations \cite[Theorem B]{DAli2025}.
Proving the existence of a Koszul filtration essentially boils down to identifying a family of linearly generated ideals which is closed under computing certain colon ideals. To this end we introduce \emph{strong G-sets} which are sets of variables in a polynomial ring $S$ that behave well in relation to a given Gröbner basis $G \subset S$ (a precise definition is given in Section \ref{sec_Gset}). Our main results in this direction are Theorems \ref{thm_Gset} and \ref{thm_strong_colon}, which generalize \cite[Theorems 1.1 and 1.3]{Ene_2015}.  Theorem \ref{thm_strong_colon} focuses on the case of binomial ideals, and states in a shortened version as follows.

\begin{thm}\label{thm_B}
Let $G$ be a reduced Gröbner basis of quadratic binomials having terms of disjoint support. If $X$ and $X \cup \{x \}$, with $x\not\in X$, are strong $G$-sets then ${(G\cup X):x}=(G\cup X')$ where $X'$ is also a strong $G$-set. 
\end{thm}

Applying Theorem \ref{thm_B} we recover two classes of ideals that define algebras with Koszul filtrations if they have a quadratic Gröbner basis. 
The first, presented in Theorem \ref{thm_GDI}, is that of \emph{generic determinantal ideals}, in which case an alternative proof can be found in \cite[Theorem 5.4(i)]{DAli2025}.
The second class, presented in Theorem \ref{thm_BEI}, is that of \emph{binomial edge ideals} for which the result was first proved in \cite[Theorem 7.55]{HHO2018}. 
Further, we conjecture that this implication holds more generally:

\begin{con}
A reduced Gröbner basis of quadratic binomials having terms of disjoint support defines an algebra with a Koszul filtration. 
\end{con}

The paper starts with some notation and preliminaries on Koszul algebras in Section \ref{sec_prelim}. Section \ref{sec_Gset} introduces (strong) $G$-sets and $G$-sequences and proves their key properties in Theorems \ref{thm_Gset} and \ref{thm_strong_colon}. In Section \ref{sec_binomial} we use the tools developed in Section \ref{sec_Gset} to construct Koszul filtrations for some algebras defined by binomial ideals. We turn to developing algorithms for finding Koszul filtrations more generally in Section \ref{sec_alg}. Implementations of these are given in Appendix \ref{sec_appendix}, and are used for computations throughout the paper. Indeed, these algorithms form the basis for establishing both existence and non-existence of Koszul filtrations. Finally, Section \ref{sec_b4t} is devoted to proving Theorem \ref{thm_A} above. 

\section{Preliminaries}\label{sec_prelim}
\noindent
Let $\K$ be a field and $S=\K[x_1,\ldots,x_n]$ be the (commutative) polynomial ring in variables $x_1,\ldots,x_n$. Throughout this paper, $R=\bigoplus_{i\geq 0} R_i$ denotes a standard graded $\K$-algebra of embedding dimension $n$; any such algebra is isomorphic to $S/I$ for a homogeneous ideal $I\subset (x_1,\ldots,x_n)^2$. We refer to the ideal $I$ as the defining ideal of $R$, and denote by $\m_R$ the homogeneous maximal ideal of $R$. Here we briefly recall some aspects of Gröbner bases, Koszul algebras, and binomial ideals.

\begin{num}[Gröbner bases]\label{par_GB}
Consider any term order on $S$. For a nonzero $f\in S$ and an ideal $\fI\subseteq S$, denote by $\init(f)$ and $\init(\fI)$ the initial term of $f$ and the initial ideal of $\fI$, respectively. Recall, for example from \cite[Chapter 1]{Sturmfels1996}, that a finite set $G\subseteq \fI$ is a \emph{Gröbner basis} for $\fI$ if $\init(\fI)=(\init(g))_{g\in G}$, and in particular $\fI=(G)$. The Gröbner basis $G$ is \emph{reduced} if all $g\in G$ are monic, and for any distinct $g,g'\in G$ no term of $g'$ is divisible by $\init(g)$. Every ideal has a unique reduced Gröbner basis w.\,r.\,t.\ the given term order.  Throughout, we tacitly assume that Gröbner bases consist of homogeneous polynomials. For more background on Gröbner bases, see for example \cite{Froberg1997} or \cite{HHO2018}.

Two term orders used frequently are the \emph{lexicographic} and \emph{degree reverse lexicographic} term orders induced by $x_1>\cdots >x_n$. In the lexicographic term order $x_1^{\alpha_1}\cdots x_n^{\alpha_n}>x_1^{\beta_1}\cdots x_n^{\beta_n}$ if for the smallest index $i$ with $\alpha_i \ne \beta_i$ one has $\alpha_i>\beta_i$. 
In the degree reverse lexicographic term order we instead have $x_1^{\alpha_1}\cdots x_n^{\alpha_n}>x_1^{\beta_1}\cdots x_n^{\beta_n}$ if $\sum_{i=1}^n\alpha_i>\sum_{i=1}^n\beta_i$, or if $\sum_{i=1}^n\alpha_i=\sum_{i=1}^n\beta_i$ and, for the largest index $i$ with $\alpha_i\not=\beta_i$, one has $\alpha_i<\beta_i$. 
\end{num}

\begin{num}[Koszul algebras]\label{par_Koszul}
A standard graded $\K$-algebra $R$ is called \emph{Koszul} if the $R$-module $\K=R/\m_R$ has a linear minimal free resolution, see for example \cite{conca2012koszulalgebrasregularity}. Gröbner bases provide a useful way to show that an algebra is Koszul. Recall that $R$ is \emph{G-quadratic} if---after some linear change of coordinates of $S$---the image of the defining ideal $I$ has a quadratic Gröbner basis. If an algebra is G-quadratic, then it is Koszul, by Fröberg \cite{Froberg1975} and Anick \cite{Anick1986}, see also \cite[Remark 1.13]{Conca2014}.

Another way to verify an algebra is Koszul is to exhibit a certain type of filtration. 
A \emph{Koszul filtration} of $R$ is a set $\mathcal{F}$ of ideals satisfying:
\begin{enumerate}
    \item[\tagKFlin] Every ideal in $\mathcal{F}$ is generated by elements of $R_1$.
    \item[\tagKFmax] The maximal ideal $\m_R$ belongs to $\mathcal{F}$.
    \item[\tagKFcol] For every nonzero ideal $\fa\in \mathcal{F}$, there exists an ideal $\fb\in \mathcal{F}$ properly contained in $\fa$ such that $\fa/\fb$ is cyclic and $\fb:\fa\in \mathcal{F}$.
\end{enumerate}
Recall here that for ideals $\fa,\fb\subseteq R$, the colon ideal is $\fb:\fa=\{f\in R \mid f\fa \subseteq \fb\}$. Observe that a Koszul filtration contains the zero ideal---this is a consequence of \tagKFmax\ and \tagKFcol---and so the definition given here is equivalent to the original one given by Conca, Trung, and Valla in \cite{CTV2001}, which was in turn motivated by work of Herzog, Hibi, and Restuccia \cite{HHR2000}. The importance of Koszul filtrations stems from the observation that the existence of a Koszul filtration implies that the algebra is Koszul \cite[Proposition 1.2]{CTV2001}. 

We say that a Koszul filtration $\cF$ is a \emph{monomial Koszul filtration} if each ideal in $\cF$ is generated by variables.

By definition, every Koszul filtration must contain a \emph{linear flag}, meaning a chain of ideals $(0)=\fa_0\subset \fa_1\subset \fa_2\subset \cdots \subset \fa_n=\m_R$, where $\fa_j$ is minimally generated by $j$ elements of $R_1$, and such that for all $j$ the colon ideal $\fa_j:\fa_{j+1}$ is generated by elements of $R_1$. Linear flags were defined in \cite{Ene_2015}. A \emph{Gröbner flag} is a linear flag which is itself a Koszul filtration. This is a very special Koszul filtration; indeed, the existence of a Gröbner flag implies that $R$ is G-quadratic \cite[Lemma 3.10]{conca2012koszulalgebrasregularity}.
\end{num}

\begin{num}[Binomial ideals]\label{par_binomial}
Our study of algebras having Koszul filtrations will place a special emphasis on algebras defined by binomial ideals. A \emph{binomial} ideal is an ideal that is generated by (homogeneous) polynomials consisting of at most two terms. An important class of binomial prime ideals is that of the toric ideals: For an $m\!\times\! n$ matrix of integers $\{a_{ij}\}$, the associated \emph{toric ideal} is the kernel of the homomorphism $\varphi:\K[x_1,\ldots,x_n]\to \K[t_1^{\pm 1}\!,\ldots,t_m^{\pm1}]$ defined by $x_j\mapsto t_1^{a_{1j}}\!\cdots t_m^{a_{mj}}$. A generating set for the toric ideal consists of binomials $u-v$, with $\varphi(u)=\varphi(v)$, such that $\gcd(u,v)=1$; see \cite[Lemma 4.1]{Sturmfels1996}. 

A polynomial has \emph{terms of disjoint support} if any pair of distinct terms $u$ and $v$ satisfy $\gcd(u,v)=1$. As seen in \cite[Corollary 4.4]{Sturmfels1996} every toric ideal has a reduced Gröbner basis consisting of binomials each having terms of disjoint support, given any term order. However, binomials with this property may form a Gröbner basis for an ideal which is not toric. For example, the ideal $(x_1x_2-x_3x_4,\,x_3x_5-x_1x_6)$ in $\K[x_1,\ldots,x_6]$ is not prime and hence not toric. On the other hand, the generators of this ideal clearly have terms of disjoint support. Moreover, the given generating set is a Gröbner basis with respect to the degree reverse lexicographic term order, so the quotient ring is G-quadratic and hence Koszul.
\end{num}

\section{Gröbner bases and linear colon ideals}\label{sec_Gset}
\noindent
Let $S=\K[x_1,\ldots,x_n]$ and $R=S/I$ be as above, and let $G$ be a Gröbner basis for $I$ with respect to some term order. In this section, we consider a certain condition on $G$ in relation to a set of linear forms, which allows us to generalize results on linear colon ideals from \cite{Ene_2015}.
To do this, we first introduce two types of sets of linear forms: $G$-sets and $G$-sequences. 

Given a Gröbner basis (of linear forms) $X \subset S$ and a polynomial $p \in S$ we let $\nf_X(p)$ denote the \emph{normal form} of $X$ modulo $p$, that is, the unique remainder of $p$ after reduction by $X$. For a set $P$ of polynomials in $S$, let
\[
\nf_X(P) = \{ \nf_X(p) \ | \ p \in P \ \text{and} \ \nf_X(p) \ne 0 \}.
\]
Notice also that, if $X$ is a reduced Gröbner basis of linear forms in $S$, then its elements have pairwise distinct leading terms. In particular, the cardinality of such a set $X$ is at most $n$. 

\begin{defn}\label{dfn_Gset}
Let $G \subseteq S$ be a Gröbner basis w.\,r.\,t.\ a term order $>$. A set $X$ of linear forms in $S$ is called a \emph{$G$-set} if $X$ itself and $\nf_X(G) \cup X$ both are Gröbner bases w.\,r.\,t.\ the term order $>$.
\end{defn}
Note that, in particular $\nf_X(G) \cup X$ is a Gröbner basis for $(G \cup X)$, as the two sets generate the same ideal. 

\begin{defn}\label{dfn_Gsequence}
Let $G$ be a Gröbner basis. A $G$-set $X$ is called a \emph{$G$-sequence} if there exists a chain of strict containments
\[ X = X_0 \supset X_1 \supset X_2 \supset \dots \supset X_m= \varnothing \]
where $m=|X|$ and each $X_j$ is a $G$-set. 
\end{defn}
Note that the $G$-sets $X_j$ are implicitly also $G$-sequences. 

\begin{exam}\label{ex_Gsequence}
The set
\[G=\{x_{1}x_{4}-x_{2}x_{3},\, x_{2}x_{5}-x_{3}x_{4},\, x_{3}x_{6}-x_{4}x_{5},\,x_{2}x_{6}-x_{4}^{2},\,x_{1}x_{6}-x_{3}x_{4},\,x_{1}x_{5}-x_{3}^{2}\} \]
is a Gröbner basis for a toric ideal w.\,r.\,t.\ the lexicographic term order such that $x_1 > \dots > x_6$. The set 
\[
\nf_{\{x_3\}}(G) \cup \{x_3\} = \{x_{1}x_{4},\, x_{2}x_{5},\, -x_{4}x_{5},\,x_{2}x_{6}-x_{4}^{2},\,x_{1}x_{6},\,x_{1}x_{5},\, x_3\} 
\]
is a Gröbner basis, and so is 
\[
\nf_{\{x_3, x_4\}}(G) \cup \{x_3, x_4\} = \{x_{2}x_{5},\,x_{2}x_{6},\,x_{1}x_{6},\,x_{1}x_{5},\, x_3, \, x_4\}.
\]
For any $\{x_3, x_4\} \subseteq X \subseteq \{ x_1, \ldots, x_6\}$ the set $\nf_X(G) \cup X$ will consist of monomials and hence be a Gröbner basis. In particular $\{ x_1, \ldots, x_6\}$ is a $G$-sequence.
\end{exam}

Next, we introduce a special type of $G$-set. 

\begin{defn}\label{def_strongG-set}
Let $G$ be a Gröbner basis. 
A set $X \subseteq \{x_1, \ldots , x_n \}$ is called a \emph{strong $G$-set} if whenever $x\in X$, $g \in G$, and $x | \init(g)$ then $g \in (X)$. 
\end{defn}

\begin{lem}\label{lem_strong_is_Gset}
Strong $G$-sets are $G$-sets.
\end{lem}
\begin{proof}
Let $G$ be a Gröbner basis, and let $X \subseteq \{x_1, \ldots , x_n \}$ satisfy  Definition \ref{def_strongG-set}. We want to prove that $\nf_X(G) \cup X$ is a Gröbner basis. We shall begin by proving that $G \cup X$ is a Gröbner basis. 

It is clear that the S-polynomial $S(g_1,g_2)$, for $g_1, g_2 \in G$, reduces to zero modulo $G$ as $G$ is a Gröbner basis. Also, $S(x,y)=0$ for any $x, y \in X$. Now consider the polynomial $S(g,x)$ for $g \in G$ and $x \in X$. If $x$ does not divide $\init(g)$ then $S(g,x)$ reduces to zero modulo $\{g,x\}$. For a proof of this fact, see e.\,g.\ \cite[Lemma 1.27]{HHO2018}.
If instead $x$ does divide $\init(g)$ then, by Definition \ref{def_strongG-set}, each term of $g$ is divisible by a variable in $X$. Hence $S(g,x)$ reduces to zero modulo $X$. This proves that $G \cup X$ is a Gröbner basis. 
Reducing the members of $G$ modulo $X$ produces a new Gröbner basis. Hence $\nf_X(G) \cup X$ is a Gröbner basis. 
\end{proof}

\begin{exam}\label{ex_Gset}
Consider the degree reverse lexicographic term order on $S$ induced by $x_1>\cdots >x_n$. Let $G$ be a Gröbner basis of polynomials. For each $j=1,\ldots,n$, the set $\{x_j,x_{j+1},\ldots,x_n\}$ is a strong $G$-set, and hence also a $G$-sequence. This follows because if $x_p$ divides the initial term $\init(g)$ of some $g\in G$ then each term in $g-\init(g)$ must be divisible by some $x_q$ with $q\geq p$. Moreover, $\{x_{j-1}\}$ is a strong $G'$-set, where $G'=\nf_{\{x_j,\ldots,x_n\}}(G)\cup \{x_j,\ldots,x_n\}$ by similar reasoning.
\end{exam}

The following lemma is inspired by \cite[Lemma 12.1]{Sturmfels1996}, but here we avoid fixing a particular term order. 

\begin{lem}\label{lem_GB_of_colon}
Let $G$ be a Gröbner basis and $X$ be a $G$-set. If $x\not\in (X)$ is a variable such that $\{x\}$ is a strong $G'$-set, where $G'=\nf_X(G)\cup X$, then the set
\begin{equation}\label{eq_GB_for_colonideal}
\left\{f\in \nf_X(G) \ \Big\vert \:x\ndiv f\right\}\cup\left\{f/x \ \Big\vert\:  f\in \nf_X(G) \text{ and }x | f\right\} \cup X
\end{equation}
is a Gröbner basis for $(G\cup X):x$.
\end{lem}
\begin{proof}
Set $\fI=(G\cup X)$ and $G'=\nf_X(G)\cup X$.  Since $X$ is a $G$-set, $G'$ is a Gröbner basis for $\fI$.  Let $G''$ be the set \eqref{eq_GB_for_colonideal}. Observe that $G''\subseteq\fI:x$ as for any $g\in G''$ either $g \in \nf_X(G)\cup X$ or $xg \in \nf_X(G)$.

We proceed to show that $G''$ is a Gröbner basis for ${\fI:x}$.
Let $h\in \fI:x$.  
Thus $\init(x h)=x\init(h)$ is divisible by $\init(f)$ for some $f\in G'$.
Since $\{x\}$ is a strong $G'$-set, if $x|\init(f)$, then $f\in (x)$, that is, $x|f$. Thus if $x\ndiv f$, then $f\in G''$ and $x\ndiv \init(f)$ so $\init(f)|\init(h)$. Otherwise, if $x | f$, then $f/x \in G''$ and $\init(h)=\init(x h)/x$ is divisible by $\init(f)/x=\init(f/x)$. In both scenarios, $\init(h)\in (\init(g))_{g\in G''}$, thus $\init(\fI:x)\subseteq (\init(g))_{g\in G''}\subseteq \init(\fI:x)$ and equality follows.
\end{proof}
Knowing a Gröbner basis of $(G\cup X):x$ as in Lemma \ref{lem_GB_of_colon} allows us to prove the following theorem. 

\begin{thm}\label{thm_Gset}
Let $G$ be a quadratic Gröbner basis and $X$ be a $G$-set. If $x\not\in (X)$ is a variable such that $\{x\}$ is a strong $G'$-set, where $G'=\nf_X(G)\cup X$, then one has 
$(G\cup X):x=(G\cup X')$
where $X'$ is a $G$-set.
\end{thm}
\begin{proof}
By definition, the set $G'=\nf_X(G) \cup X$ is a Gröbner basis for $(G \cup X)$,  and by Lemma \ref{lem_GB_of_colon}
\[
G''=\left\{f\in \nf_X(G) \ \Big\vert \:x\ndiv f\right\}\cup\left\{f/x \ \Big\vert\:  f\in \nf_X(G) \text{ and }x | f\right\} \cup X
\]
is a Gröbner basis for $(G\cup X):x$. We need to prove that the set
\[
X'\colonequals \left\{f/x \ \Big\vert\:  f\in \nf_X(G) \text{ and }x | f\right\} \cup X
\]
is a $G$-set. Since $X$ is a set of linear forms, and $G$ is a set of quadratic forms, $X'$ consists of linear forms. Observe that $(G'')=(G' \cup X')$ and that $G' \cup X' \supseteq G''$. In general, if $H_1$ is a Gröbner basis and $H_2 \supseteq H_1$ generates the same ideal, then $H_2$ is also a Gröbner basis. Since $G''$ is a Gröbner basis, so is $G' \cup X'$.  The set $X'$ is also a Gröbner basis, as it is the set of all linear elements of the Gröbner basis $G'\cup X'$. Reducing the members of $G'$ modulo $X'$ produces a new Gröbner basis. Since $X \subseteq  X'$ we have $\nf_{X'}(G') = \nf_{X'}(G)$, and we have proved that $\nf_{X'}(G) \cup X'$ is a Gröbner basis.
\end{proof}

Theorem 1.1 in \cite{Ene_2015} states that if $G$ is a quadratic Gröbner basis w.\,r.\,t.\  the degree reverse lexicographic term order then the ideals $(G \cup \{x_{j}, \ldots, x_n\}):x_{j-1}$ are generated by $G$ together with a set of linear forms. We saw in Example \ref{ex_Gset} that the sets $\{ x_{j}, \ldots, x_n\}$ are (strong) $G$-sets and $\{x_{j-1}\}$ is a strong $G'$-set where $G'=\nf_{\{x_j,\ldots,x_n\}}(G)\cup \{x_j,\ldots,x_n\}$, so \cite[Theorem 1.1]{Ene_2015} is now recovered as a corollary of Theorem \ref{thm_Gset}. The authors also give a more precise description of the set of linear forms obtained in the case $G$ is a reduced Gröbner basis of binomials with terms of disjoint support, see \cite[Theorem 1.3]{Ene_2015}. We proceed with a generalization of that result. 

\begin{thm}\label{thm_strong_colon}
Let $G$ be a reduced Gröbner basis of quadratic binomials having terms of disjoint support. Let $X$ and $X \cup \{x \}$ be strong $G$-sets with $x\not\in X$. Then ${(G\cup X):x}=(G\cup X')$
where  
\[
X'= \left\{ \init(g)/x \ \Big| \ g \in G \ \text{and } x|\init(g)  \right\} \cup X
\]
is a strong $G$-set. 
\end{thm}

\begin{proof}
Write $G=\{u_1-v_1,\ldots,u_t-v_t\}$, where $\init(u_i-v_i)=u_i$ for each $i=1,\ldots,t$. Note that some of the $v_i$ may be equal to $0$. Since $X$ is a strong $G$-set we may reorder the elements of $G$ so that there are integers $r\leq s\leq t$ such that $v_i\not\in (X)$ if and only if $1\leq i \leq r$, and $u_j\not\in (X)$ if and only if $1\leq j\leq s$. That is, the normal form $\nf_{X}(u_i-v_i)$ equals $u_i-v_i$ for $1\leq i \leq r$ (with $u_i,v_i\neq0$), equals $u_i$ for $r+1\leq i\leq s$, and equals $0$ for $i>s$. The set
$$G'=\{u_1-v_1,\ldots,u_r-v_r,u_{r+1},\ldots,u_s\} \cup X$$
is thus a Gröbner basis for $(G\cup X)$ by Lemma \ref{lem_strong_is_Gset} and the definition of $G$-set. Observe also that $\{x\}$ is a strong $G'$-set: Suppose $f=\nf_X(g)$ for some $g\in G$ and that $x|\init(f)$. Either $\init(f)\not=\init(g)$, in which case $\init(g)$ was divisible by some variable of $X$, or $\init(f)=\init(g)$. In either case, since $X\cup \{x\}$ is a strong $G$-set, one has $g\in (X\cup \{x\})$ hence $f\in (x)$.
Using this, along with Lemmas \ref{lem_strong_is_Gset} and \ref{lem_GB_of_colon}, it follows that there is an equality
\begin{equation}\label{colon}
    (G\cup X):x=( \{ u_1-v_1,\ldots,u_r-v_r \} \cup \{u_j\}_{j\in \Gamma} \cup \{u_j/x\}_{j\in \Lambda} \cup X),
\end{equation}
where $\Gamma=\{j\mid r+1\leq j\leq s\text{ and } x\ndiv u_j\}$ and $\Lambda=\{j\mid r+1\leq j\leq s\text{ and } x| u_j\}$, and the generating set on the right is a Gröbner basis. 

Set $X'\colonequals X\cup \{u_j/x\}_{j\in \Lambda}$, and note that $(G\cup X):x=(G\cup X')$. One has that $X'\subseteq \{x_1,\ldots,x_n\}$ as $G$ is a quadratic Gröbner basis. It remains to show that $X'$ is a strong $G$-set. Suppose that $y\in X'$ divides some $u_i$. We want to show that $v_i$ is divisible by an element of $X'$. If $y\in X$ there must exist $z\in X\subseteq X'$ such that $z|v_i$ since $X$ is a strong $G$-set. Suppose now that $y = u_j/x$ for some $j\in \Lambda$. If $i=j$, then the fact that $x$ divides $u_i$ implies that there exists $z\in X$ such that $z|v_i$, since $X \cup \{x\}$ is a strong $G$-set and $u_i$ and $v_i$ have disjoint support.  
If instead $i\neq j$, then $u_j$ and $u_i$ share a factor $y$ and their respective binomials can be written as
$$
\left\{\begin{array}{ccc}
   u_i-v_i&\!\!=\!\!& zy-v_i\\
   u_j-v_j&\!\!=\!\!&xy-v_j
\end{array}\right.
$$
where $v_j$ is divisible by some variable in $X$ (since $r+1\leq j\leq s$). Note that the second equality thus implies $xy\in (G\cup X)$. It now follows from the equality
$$x(u_i-v_i)=xzy-xv_i$$
that $xv_i\in (G\cup X)$, and so $v_i\in (G\cup X):x$. We know that $v_i$ is not divisible by $u_k$ for $k=1,\ldots,s$ since $G$ is a reduced Gröbner basis. From \eqref{colon}, we conclude that $v_i$ is divisible by one of $\{u_j/x\}_{j\in\Lambda} \cup X=X'$, so $X'$ is a strong $G$-set.
\end{proof}

Say that a set is a \emph{strong $G$-sequence} if it is a $G$-sequence consisting of strong $G$-sets. In a situation where every strong $G$-set is a strong $G$-sequence, then Theorem \ref{thm_strong_colon} can be used to produce a Koszul filtration of $S/(G)$. 
We end the section, however, by pointing out that even if $X$ and $X\cup \{x\}$ are strong $G$-sequences, in general the resulting set $X'$ in Theorem \ref{thm_strong_colon} need not be a strong $G$-sequence.
\begin{exam}
Let $S=\Q[x,y_1,y_2,z_1,\ldots,z_4,t]$ and consider
$$G=\{xy_1-z_3t,\, xy_2-z_4t,\, y_1z_1-y_2z_2,\, y_2z_3-y_1z_4,\, z_1z_3-z_2z_4\}.$$ 
This is a reduced quadratic Gröbner basis in the degree reverse lexicographic order such that $x>y_1>y_2>z_1>\cdots>z_4>t$. Indeed, the ideal $(G)$ is even toric, and in particular the generators are binomials with terms of disjoint support.

The sets $\{x,t\}$ and $\{t\}$ are strong $G$-sets, hence strong $G$-sequences. However, one has $(G \cup \{t\}):x = (G \cup \{y_1, y_2, t\})$, but $\{y_1, y_2, t\}$ is not a strong $G$-sequence.

We remark, however, that $R=S/(G)$ does have a monomial Koszul filtration. 
\end{exam}

In order to use Theorem \ref{thm_strong_colon} for constructing Koszul filtrations we therefore need to restrict to some particular family $\cG$ of strong $G$-sequences for which the new $G$-sets $X'$ produced by Theorem \ref{thm_strong_colon} also belong to $\cG$. In the next section we illustrate this idea on two classes of algebras.

\section{Koszul filtrations for some binomial ideals}\label{sec_binomial}
\noindent
We now apply the results of the previous section to show that having a quadratic Gröbner basis implies the existence of a Koszul filtration for algebras defined by binomial edge ideals and by generic determinantal ideals.

Consider a matrix 
$$\X=\begin{pmatrix} x_{11} & x_{12} & \cdots & x_{1n} \\ \vdots &&&\vdots \\ x_{m1} & x_{m2} & \cdots & x_{mn}\end{pmatrix}$$
where each entry is a variable. Let $S=\K[\X]$ be the polynomial ring in the $mn$ variable entries of $\X$. A \emph{$t$-minor} of $\X$ is a determinant of a $t \times t$ submatrix. 
A term order on $S$ is called \emph{diagonal} if the initial term of any $t$-minor is the product of the main diagonal entries of the corresponding $t\times t$ submatrix. This can be obtained for instance by taking the degree reverse lexicographic order with
$$x_{m1}>\cdots >x_{21}>x_{11} > x_{m2}>\cdots > x_{22}>x_{12}>\cdots> x_{1n},$$
or equivalently, $x_{ij}>x_{kl}$ if $j<l$, or $j=l$ and $i>k$. For the remainder of this section we fix this term order. 

Let $G$ be a Gröbner basis consisting of some subset of the $2$-minors of $\X$. Let $\cG$ be the family (depending on $G$) of sets $X$ of variables such that
\begin{equation}\label{det_G-sets}
    \text{
    if} \ \begin{vmatrix}
        x_{ik} & x_{i \ell} \\
        x_{jk} & x_{j \ell}
    \end{vmatrix} \in G \text{ and $x_{ik}$ or $x_{j\ell}$ belongs to $X$ then $x_{i\ell}\in X$}
\end{equation}
where $i<j$ and $k<\ell$. Note that each such set $X$ is a strong $G$-set.

\begin{lem}\label{lem_det_G-seq}
    If $X\in \cG$ and $x$ is the largest member of $X$ then $X \setminus \{x\} \in \cG$.
\end{lem}
\begin{proof}
    Note that in \eqref{det_G-sets} $x_{i\ell}$ is smaller than both $x_{ik}$ and $x_{j\ell}$. Hence \eqref{det_G-sets} still holds after removing the largest element of the set $X$. 
\end{proof}

Let $I_2(\X)$ denote the ideal generated by all $2$-minors of $\X$. The ideal $I_2(\X)$ is toric \cite[Proposition 4.1.3]{BCR2022} and the $2$-minors form a (quadratic) Gröbner basis for $I_2(\X)$ under any diagonal term order \cite[Theorem 4.1.1]{BCR2022}. Thus the determinantal ring $S/I_2(\X)$ is Koszul \cite[Theorem 4.1.2]{BCR2022}.  
The following shows that in fact this ring has a monomial Koszul filtration. This was also recently shown in \cite[Theorem 5.4(i)]{DAli2025}, but the filtration we give here is smaller.

\begin{thm}\label{thm_GDI}
The determinantal ring $S/I_2(\X)$ has a monomial Koszul filtration.
\end{thm}
\begin{proof}
    Let $G$ be the Gröbner basis consisting of all $2$-minors of $\X$. We shall prove that the ideals of $R=S/I_2(\X)$ generated by the canonical images of sets in $\cG$ form a Koszul filtration. It is clear that the ideals satisfy \tagKFlin, and since  $\cG$ contains the set of all variables \tagKFmax \ holds as well. Further, let $X\in \cG$ and let $X'=X \setminus \{x\}$ where $x$ is the largest member of $X$. We shall prove that $(G \cup X'):x=(G \cup X'')$ where $X'' \in \cG$. This together with Lemma \ref{lem_det_G-seq} then shows that the last condition \tagKFcol \ holds. 

   Let $x_{ik}$ be the largest member of $X$, and $X'=X\setminus\{x_{ik}\}$. By Theorem \ref{thm_strong_colon}
   \[
X''= \left\{ \init(g)/x_{ik} \ \Big| \ g \in G \ \text{and } x_{ik}|\init(g)  \right\} \cup X'.
\]
We need to prove that this set satisfies \eqref{det_G-sets}. By Lemma \ref{lem_det_G-seq} \eqref{det_G-sets} holds for the subset $X'$. Consider therefore $x_{j \ell}=\init(g)/x_{ik}$ for $g \in G$. That is $g=x_{ik}x_{j \ell}-x_{i \ell}x_{jk}$ is a $2$-minor. Now assume $x_{j \ell}$ appears on the main diagonal of another $2 \times 2$ submatrix. Then we have a $2$-minor $h=x_{j \ell}x_{rs}-x_{js}x_{r \ell}$ with initial term $x_{j \ell}x_{rs}$. We now need to consider four cases depending on where on the main diagonal $x_{j \ell}$ appears on the two submatrices corresponding to $g$ and $h$.
\begin{enumerate}
    \item Suppose $i<j$ and $k< \ell$ together with $j<r$ and $\ell<s$. In this case we want to prove that $x_{js} \in X''$. Since $i<j$ and $k<s$ the minor given by the submatrix of rows $i,j$ and columns $k,s$ have leading term $x_{ik}x_{js}$. This shows that  $x_{js} \in X''$.
    \item Suppose $i>j$ and $k> \ell$ together with $j<r$ and $\ell<s$. We want to prove that $x_{js} \in X''$. If $k>s$ then, since $i>j$, the term $x_{js}x_{ik}$ is the lead term of a $2$-minor, and hence $x_{js} \in X''$. Assume instead that $k \le s$. Considering the $2$-minor
    \[
    g=\begin{vmatrix}
        x_{j \ell} & x_{j k} \\
        x_{i \ell} & x_{ik}
    \end{vmatrix} 
    \]
    and that $x_{ik} \in X$ gives that $x_{jk} \in X'$. If $k=s$ we have $x_{js} \in X' \subseteq X''$. If $k<s$ we consider the $2$-minor
    \[
    \begin{vmatrix}
        x_{j k} & x_{j s} \\
        x_{i k} & x_{is}
    \end{vmatrix}.
    \]
    Since $x_{jk} \in X'$ it follows that $x_{js} \in X' \subseteq X''$.
    \item Suppose $i<j$ and $k< \ell$ together with $j>r$ and $\ell>s$. In this case we want to prove that $x_{r \ell} \in X''$. If $r>i$ then, since $k< \ell$, the term $x_{r \ell}x_{ik}$ is a valid lead term of a $2$-minor, and $x_{r \ell} \in X''$. Assume $r \le i$ and consider the submatrix
    \[
    \begin{pmatrix}
        x_{rk} & x_{r \ell} \\
        x_{ik} & x_{i \ell} \\
        x_{jk} & x_{j \ell}
    \end{pmatrix} \quad \text{(where possibly the rows $r$ and $i$ coincide).}
    \]
    Since $x_{ik} \in X$ we have $x_{i \ell} \in X'$, and then in turn $x_{r \ell} \in X' \subseteq X''$. 
    \item Last, suppose $i>j$ and $k> \ell$ together with $j>r$ and $\ell>s$. We want to prove that $x_{r \ell} \in X''$. Since $r<i$ and $\ell <k$ we have $x_{r \ell}x_{ik}$ as a leading term of a $2$-minor, and hence $x_{r \ell} \in X''$. 
\end{enumerate}
We have now proved that $X'' \in \cG$ which completes the proof. 
\end{proof}

We now turn to binomial edge ideals. The binomial edge ideals were independently introduced by Herzog, Hibi, Hreinsdóttir, Kahle, and Rauh \cite{HHHKR2010} and Ohtani \cite{Ohtani2011}. For this application we consider the matrix
\[
\X=\begin{pmatrix} x_1 & x_2 & \dots & x_n \\ y_1 & y_2 & \dots & y_n  \end{pmatrix}.
\]
Let $\Gamma$ be a graph with vertex set $\{1, \ldots, n\}$ and edge set $E(\Gamma)$. The \emph{binomial edge ideal} $I_\Gamma$ is the ideal of the ring $S=\K[\X]$ generated by the set $G_\Gamma$ of $2$-minors of $\X$ given by columns $i$ and $j$ such that $\{i,j\} \in E(\Gamma)$. Adopting the terminology from  \cite{HHHKR2010}, the graph $\Gamma$ is closed if, whenever $i,j,k$ are distinct, $\{i,j\}, \{i,k\} \in E(\Gamma)$, and $\{j,k\} \notin E(\Gamma)$, either $j<i<k$ or $j>i>k$. It was proved in \cite{HHHKR2010} that the graph $\Gamma$ is closed if and only if the set $G_\Gamma$ is a Gröbner basis w.\,r.\,t.\ the lexicographic term order with $x_1> \dots > x_n>y_1> \dots > y_n$. Notice that this is a diagonal term order. Choosing any other diagonal term order does not affect the initial ideal, and hence $G_\Gamma$ is a Gröbner basis  w.\,r.\,t.\ our degree reverse lexicographic order if and only if $\Gamma$ is closed.
Further, it was proved in \cite{CR2011} that $\Gamma$ is closed if and only if $G_\Gamma$ is a quadratic Gröbner basis w.\,r.\,t.\ some term order.

A complete classification of binomial edge ideals that define Koszul algebras is given in \cite{LMMP26}. 
It is proved in \cite[Theorem 7.55]{HHO2018} that the binomial edge ideal of a closed graph defines an algebra with a Koszul filtration. Hence having a quadratic Gröbner basis implies the existence of a Koszul filtration for binomial edge ideals. We now illustrate how this result can be recovered by applying Theorem \ref{thm_strong_colon}. 

\begin{thm}\label{thm_BEI}
    If a binomial edge ideal $I_\Gamma$ has a quadratic Gröbner basis then the ring $S/I_\Gamma$ has a monomial Koszul filtration. 
\end{thm}
\begin{proof}
As $G_\Gamma$ is a Gröbner basis, the vertices of $\Gamma$ are numbered so that $\Gamma$ is closed. Let $\cG$ be the family of subsets $X$ of variables satisfying \eqref{det_G-sets} w.\,r.\,t.\ the Gröbner basis $G_\Gamma$. In the context of binomial edge ideals,  \eqref{det_G-sets} reads
\begin{equation}\label{eq_BEI_G-sets}
    \text{
    If $\{i,j\} \in E(\Gamma)$ with $i<j$ and $x_i$ or $y_j$ belongs to $X$ then $x_j \in X$.
    }
\end{equation}

Let $X'=X\setminus \{z\}$ where $z$ is the largest element of $X$. By Lemma \ref{lem_det_G-seq}, $X' \in \cG$. We now show that $(G_\Gamma\cup X'):z=(G_\Gamma\cup X'')$ for some $X''\in \cG$. Suppose first that $z=x_a$ for some $a$. Theorem \ref{thm_strong_colon} yields
$$ X''=\{y_j\mid a<j\text{ and }\{a,j\}\in E(\Gamma)\} \cup X'\;.$$
Suppose $y_j\in X''$ such that $\{a,j\}\in E(\Gamma)$ with $a<j$. We know that $x_a\in X$ and since $X$ satisfies \eqref{eq_BEI_G-sets} we get $x_j\in X'$. To conclude that $X'' \in \cG$, we must also check that $y_a \notin X''$ and $x_i \notin X''$ for $i<a$ and $\{i,a \} \in E(\Gamma)$, as this would require $x_a \in X''$. This follows from the fact that $x_a$ was the largest member of $X$. 

Suppose instead that $z=y_a$. Theorem \ref{thm_strong_colon} gives us
$$X''=\{x_i \mid i<a \text{ and }\{i,a\}\in E(\Gamma)\} \cup X'\;.$$
 Suppose that $x_i\in X''$ such that $i<a$ and $\{i,a\}\in E(\Gamma)$. If $\{i,j\}\in E(\Gamma)$ with $i<j$ we need to prove that $x_j \in X''$.
 Closedness of $\Gamma$ implies there is an edge $\{a,j\}$ with either $a<j$ or $j<a$. In the first case, since $x_a\in X$, we get that $x_j\in X'\subseteq X''$. In the second case, we know $x_j\in X''$ by definition. Thus $X''\in \cG$.

Finally, observe that $\{x_1,\ldots,x_n,y_1,\dots,y_n\} \in \cG$. We have proved that the family of ideals of $R=S/I_\Gamma$ generated by the canonical images of the sets in $\cG$ is a monomial Koszul filtration. 
\end{proof}

As we see in the next example, not every quadratic Gröbner basis of binomials with terms of disjoint support defines an algebra with a monomial Koszul filtration. 

\begin{exam}
Let $R=\Q[x_1, \ldots, x_6]/(H)$ where 
\[
H=\{x_1x_4-x_2x_3,\, x_2x_5-x_3x_4,\, x_3x_6-x_4x_5 \}\;.
\]
The set $H$ is a Gröbner basis w.\,r.\,t.\ the lexicographic term order with $x_1> \dots > x_6$.
All the colon ideals $(0):x_i \subset R$ have generators of degree higher than 1, and hence there is no monomial Koszul filtration of $R$. Allowing other linear forms, a Koszul filtration for $R$ is given by
\begin{align*}
\{ &(0), \: (x_1+x_2), \: (x_1,\, x_2), \: (x_1+x_2,\, x_3+x_4), \: (x_1+x_2,\, x_3+x_4,\, x_5+x_6), \\ 
&(x_1+x_2,\, x_3,x_4), \: (x_1,\, x_2,\, x_3+x_4,\, x_5+x_6), \: (x_1,\ldots ,x_4,\, x_5+x_6), \: (x_1,\ldots,x_6) \}
\end{align*}
It is remarkable that when extending $(H)$ to the toric ideal $(G)$ given in Example \ref{ex_Gsequence}, there is a monomial Koszul filtration. Indeed 
\begin{align*}
    \{  &(0), \: (x_6), \:(x_4,\,x_6),\: (x_5,\,x_6),\: (x_2,\,x_4,\, x_6),\:  (x_4,\,x_5,\, x_6),\\ &(x_3, \ldots, x_6),\:
    (x_2, \ldots, x_6) ,\: (x_1, \ldots, x_6)\}
\end{align*}
is a Koszul filtration of $\Q[x_1, \ldots, x_6]/(G)$. 
\end{exam}

Based on our results and observations in Sections \ref{sec_Gset} and \ref{sec_binomial}, we pose the following conjecture. 

\begin{con}\label{conjecture}
Let $G$ be a reduced Gröbner basis of quadratic binomials each having terms of disjoint support. Then $S/(G)$ has a Koszul filtration. Moreover, if $(G)$ is toric then $S/(G)$ has a monomial Koszul filtration.
\end{con}

\section{Algorithms for constructing filtrations}\label{sec_alg}
\noindent
As above, let $R=S/I$ be a standard graded $\K$-algebra. In this section, we present an algorithm for constructing a Koszul filtration of $R$ built from finitely many linear forms, if one exists. A first step is to construct a weaker type of filtration, which we now define.

\begin{defn}
A \emph{linear filtration} of $R$ is a set $\mathcal{L}$ of ideals generated by elements of $R_1$, such that $\cL$ contains $\m_R$ and satisfies the condition:
\begin{enumerate}
    \item[\tagLF] For every nonzero ideal $\fa\in \mathcal{L}$, there exists an ideal $\fb\in \mathcal{L}$ properly contained in $\fa$ such that $\fa/\fb$ is cyclic and $\fb:\fa$ is generated by elements of $R_1$.
\end{enumerate}
In other words, $\mathcal{L}$ satisfies \tagKFlin, \tagKFmax, and \tagLF. Observe that \tagLF\ is weaker than \tagKFcol: it does not require that the colon ideals $\fb:\fa$ belong to $\cL$.
\end{defn}

Koszul filtrations and linear flags are examples of linear filtrations. For convenience, we refer to a set of ideals in $R$ satisfying only \tagKFlin\ and \tagLF\ as a \emph{partial linear filtration}, and a set satisfying only \tagKFlin\ and \tagKFcol\ as a \emph{partial Koszul filtration}; that is, we do not require the set to contain $\m_R$.

We start with an algorithm which constructs a partial linear filtration, where the generators of the ideals are restricted to a given finite set $L\subseteq R_1$. The algorithm begins by considering ideals $(\ell)$, for $\ell\in L$, which are added to a set $\cL$ of potential filtration members if the colon ideal $(0):\ell$ is generated by elements of $L$. 
Proceeding inductively, we record this process in the following algorithm.

\begin{algorithm}[H]
\caption{Construct a partial linear filtration}\label{alg_partial_linear}
\begin{algorithmic}[1]
\Require{A standard graded $\K$-algebra $R$ of embedding dimension $n$, and a finite subset $L\subseteq R_1$}
\Ensure{A finite partial linear filtration $\mathcal{L}$ of $R$, whose ideals are generated by elements of $L$}
\State $\mathcal{L}\colonequals\{(0)\}$
\For{$i=1$ \textbf{to} $n$}   
\For{$L'\subseteq L$ such that $|L'|=i$}
\For{$\fb\in \mathcal{L}$ such that $\fb\subset (L')$ \textbf{and} $(L')/\fb$ is cyclic}
  \If{$\fb:L'=(L'')$ for some $L''\subseteq L$}
   \State add $(L')$ to $\mathcal{L}$
\EndIf
\EndFor
\EndFor
\EndFor
\State \Return{$\mathcal{L}$}
\end{algorithmic}
\end{algorithm}

\begin{proof}
Termination of the algorithm follows from finiteness of $n$, $|L|$, and $|\cL|$. Indeed, the cardinality of $\cL$ is at most $2^{|L|}$.

The resulting set $\cL$ contains ideals generated by, and whose colon ideals are generated by, elements in $L\subseteq R_1$, thus it satisfies \tagKFlin\ and \tagLF. 
\end{proof}

\begin{rem}
If $R$ has a linear filtration, which consists of ideals generated by elements of $L$, and whose colon ideals are generated by elements of $L$, then it must be contained in the output $\cL$ of Algorithm \ref{alg_partial_linear}. It is straightforward to check whether the resulting set $\cL$ from Algorithm \ref{alg_partial_linear} is a linear filtration: one only needs to determine whether $\m_R\in \cL$.
\end{rem}

Given a linear filtration $\cL$ of $R$, although the colon ideals are generated by elements of $R_1$, there is no guarantee that the colon ideals belong to $\cL$, as required by the definition of a Koszul filtration. We therefore next trim such a linear filtration to construct a (partial) Koszul filtration. Given a finite linear filtration $\cL$, the algorithm constructs a subset $\cF\subseteq \cL$ which has the property \tagKFcol, that is, the colon ideals belong to the set $\cF$.

\begin{algorithm}[H]
\caption{Construct a partial Koszul filtration}\label{alg_partial_koszul}
\begin{algorithmic}[1] 
\Require{A finite linear filtration $\mathcal{L}$ of a standard graded $\K$-algebra $R$}
\Ensure{A partial Koszul filtration $\mathcal{F}$ of $R$}
\State{$\cF \colonequals \cL$ \textbf{and} $\cF_{\text{prev}} \colonequals \varnothing$}
\While{$\cF \neq \cF_{\text{prev}}$}
\State{$\cF_{\text{trim}}\colonequals \{(0)\}$}
\For{$\fa\in \cF$} 
\If{$\fb:\fa\in \cF$ for some $\fb\in \cF$ with $\fb\subset \fa$ and $\fa/\fb$ cyclic}
\State{add $\fa$ to $\cF_{\text{trim}}$} 
\EndIf
\EndFor
\State{$\cF_{\text{prev}}\colonequals \cF$}
\State{$\cF \colonequals \cF_{\text{trim}}$}
\EndWhile
\State \Return {$\cF$}
\end{algorithmic}
\end{algorithm}

\begin{proof}
The set $\cL$ is finite, thus the sets $\cF_{\text{prev}}$ and $\cF_{\text{trim}}$ at each iteration are finite.  It follows that the inner \textbf{for} loop is finite. The \textbf{while} loop thus has finitely many iterations, since at each step the new set $\cF$ is a proper subset of the previous set $\cF_{\text{prev}}$, until it returns the same set. Hence the algorithm terminates.

The set $\cF$ that is returned by the algorithm evidently satisfies \tagKFlin\ and \tagKFcol, and is therefore a partial Koszul filtration.  
\end{proof}

\begin{rem}
For an infinite field $\K$, these algorithms cannot in general be used to conclude the non-existence of a Koszul filtration. However, the following give some positive directions in which these algorithms may be used.
\begin{enumerate}
\item For any field $\K$, combining Algorithms \ref{alg_partial_linear} and \ref{alg_partial_koszul} yields a way to construct a partial Koszul filtration $\cF$ of $R$, which consists of ideals whose generators belong to a finite set $L\subseteq R_1$. The two algorithms will produce a Koszul filtration consisting of ideals whose generators belong to $L$, if such a Koszul filtration exists. 
Indeed, it is a Koszul filtration if $\m_R\in \cF$. 
\item If $\K$ is a finite field, then the set of linear forms is finite, so Algorithms \ref{alg_partial_linear} and \ref{alg_partial_koszul} together determine the existence of a Koszul filtration, for example by taking $L$ to be the set of all linear forms and removing scalar multiples.
\item If a Koszul filtration is obtained by Algorithms \ref{alg_partial_linear} and \ref{alg_partial_koszul}, it is maximal in the sense that it contains all possible Koszul filtrations consisting of ideals generated by elements of $L$. 
\end{enumerate}
\end{rem}

A starting point for constructing a Koszul filtration is often to first find a linear flag. We thus turn to an algorithm that searches for linear flags. Fix a finite set $L\subseteq R_1$. This algorithm starts at the maximal ideal $\m_R$, and first considers ideals $\fa_{n-1}\subset \m_R$ generated by elements of $L$ and with one fewer generator. If $\fa_{n-1}:\m$ is generated by linear forms, the algorithm stores the partial flag $\fa_{n-1}\subset \m_R$, and continues inductively. If at any step, an ideal $\fa_i$ has no ideal $\fa_{i-1}\subset \fa_i$ such that $\fa_i/\fa_{i-1}$ is cyclic and $\fa_{i-1}:\fa_i$ is generated by linear forms, then the partial flag is thrown out.

\begin{algorithm}[H]
    \caption{Find all linear flags}\label{alg_findflags}
    \begin{algorithmic}[1]
        \Require{A standard graded $\K$-algebra $(R,\m_R)$ of embedding dimension $n$, and a finite subset $L\subseteq R_1$ such that $(L)=\m$}
        \Ensure{A set $\Flags$ of all linear flags of $R$ that consists of ideals which are generated by elements of $L$}
        \State{$\Flags \colonequals \{\{\m_R\}\}$}      
        \While{$\Flags \ne \varnothing$ \textbf{and} $(0)$ is not contained in any $F\in \Flags$}
        \State$\Flags_{new}\colonequals \varnothing$
        \For{$\{\fa_i\subset\cdots\subset \fa_{n}= \m_R\}\in \Flags$}
        \For{$L'\subseteq L$ such that $|L'|=i-1$, $(L')\subset \fa_i $,  $\fa_i/(L')$ is cyclic }
        \If{ $(L'):\fa_i$ is generated by elements of $R_1$}
        \State{add $\{(L')\subset \fa_i\subset\cdots\subset\fa_n=\m_R\}$ to $\Flags_{new}$}
        \EndIf
        \EndFor
        \EndFor 
        
        \State  $\Flags:= \Flags_{new}$
        \EndWhile
     \State \Return {$\Flags$}   
     \end{algorithmic}
\end{algorithm}

\begin{proof}
Termination follows from the finiteness of $L$. Let us consider the correctness of the output. Given a set $\{\fa_i\subset \cdots \subset \fa_{n}= \m_R\}$, if there is no ideal $\fa_{i-1}\subset \fa_i$ satisfying the necessary conditions, then there is no linear flag of $R$ consisting of ideals generated by elements of $L$. On the other hand, if the \textbf{while} loop runs until the zero ideal belongs to one (and hence every) $F\in \Flags$, then the output contains all possible linear flags of $R$ consisting of ideals generated by elements of $L$.
\end{proof} 

Implementations of these algorithms in Macaulay2 are given in Appendix \ref{sec_appendix}. We end the section by applying these algorithms to the algebras considered by Roos in \cite{Roos2016} and to the pinched Veronese algebra. 

\begin{exam}\label{ex_Roos}
Let $\K$ be a field of characteristic zero and let $R_{i}=\K[x,y,z,u]/I$ be any of the 83 rings considered by Roos in \cite[Tables 1--7]{Roos2016}, 46 of which are Koszul. Using the algorithms in this section, we briefly argue here that all of the Koszul algebras have Koszul filtrations, although not all are G-quadratic. 

First, for $i$ equal to one of 1--5, 8--11, 18, 21, 23, 25--28, 38, 41, 43, 46, 49, 50, 52--54, 61, 63, 64, 66--68, 70--73, 75--83, the ring $R_i$ has a monomial Koszul filtration. This is straightforward to check with the Macaulay2 functions in \ref{M2_linear} and \ref{M2_koszul}, but for many of them this is clear for other reasons, such as being a quadratic monomial ideal \cite[Theorem 3.12]{Conca2014}. Rings $R_{45}$ and $R_{58}$, however, do not have monomial Koszul filtrations. Nevertheless, one can still find Koszul filtrations.  Initialize the functions from \ref{M2_linear} and \ref{M2_koszul} in Macaulay2 (the input is truncated here to save space):
{\footnotesize
\begin{Verbatim}[samepage=true]
i1 : partialLinearFiltration = (R,L) -> ( ... );
i2 : partialKoszulFiltration = (R,LF) -> ( ... );
\end{Verbatim}
}
\noindent
Check ring number 45:
{\footnotesize
\begin{Verbatim}[samepage=true]
i3 : R45=QQ[u,x,y,z]/ideal(x*y + y*z, x*y + z^2 + y*u, y*u + z*u, y^2, x*z);
i4 : L45={u,x,y,z,x+y,y+z};
i5 : PKF45=partialKoszulFiltration(R45,partialLinearFiltration(R45,L45));
i6 : member(trim ideal(z,y,x,u),flatten PKF45)
o6 = true
\end{Verbatim}
}
\noindent
Then check ring number 58:
{\footnotesize
\begin{Verbatim}[samepage=true]
i7 : R58=QQ[u,x,y,z]/ideal(x^2 + x*y, x^2 + z*u, y^2, z^2, x*z + y*u, x*u);
i8 : L58={u,x,y,z,2*x+y-z,2*x+y+z,2*u-y+z,2*u+y+z};
i9 : PKF58=partialKoszulFiltration(R58,partialLinearFiltration(R58,L58));
i10 : member(trim ideal(z,y,x,u),flatten PKF58)
o10 = true
\end{Verbatim}
}
\noindent
In either case, the output verifies that the ring has a Koszul filtration.

Last we note that the ring $R_{58}$ is not G-quadratic. The Hilbert series of $R_{58}$ is $\frac{1+3T-3T^3}{1-T}$. If the defining ideal $I$ of $R_{58}$ would have a quadratic Gröbner basis, in any choice of coordinates, then the initial ideal $\init(I)$ would be an ideal generated by monomials of degree 2 such that $S/\init(I)$ has the same Hilbert series as $R_{58}$. To obtain this series the ideal $\init(I)$ must have six generators and dimension one. A computation shows that no such monomial ideal has the desired Hilbert series. So we can conclude that $R_{58}$ is not G-quadratic.
\end{exam}

\begin{exam}\label{PV_example}
The pinched Veronese algebra is the $\K$-algebra $PV$ generated by all monomials of $\K[a,b,c]$ of degree~3 involving at most two variables, that is,  
$$
PV = \K[a^3, a^2b, a^2c, ab^2, ac^2, b^3, b^2c, bc^2, c^3].
$$
This algebra was shown to be Koszul by Caviglia \cite{Caviglia2009}; a more theoretical argument was later given by Caviglia and Conca \cite{CC2013}. This result was motivated by earlier work on Koszulness of Veronese subalgebras by Backelin and Fröberg \cite{BF1985}.  However, it is an open question whether the pinched Veronese $PV$ is G-quadratic \cite[Question 2.15]{conca2012koszulalgebrasregularity}. 
We fix here a presentation of $PV$ by its natural coordinates. Let $S = \K[x_1,\ldots,x_9]$ and define $\rho: S \rightarrow \K[a,b,c]$ to be the homomorphism sending $x_i$ to the $i$-th generator of $PV$. Explicitly,  
$$
\rho(x_1) = a^3,\; \rho(x_2) = a^2b,\; \ldots,\; \rho(x_9) = c^3.
$$
Then $S/I$ is isomorphic to $PV$, where $I$ denotes the toric ideal $\ker(\rho)$.
We show that $PV$ has no monomial Koszul filtration. As $I$ is prime, its reduced Gröbner basis $G$ is generated by binomials with terms of disjoint support.  Letting $L=\{x_1,\ldots ,x_9\}$, one can employ Algorithms \ref{alg_partial_linear} and \ref{alg_partial_koszul} to show that $S/I$ has no monomial Koszul filtration. This can be verified directly in Macaulay2 with the functions in \ref{M2_linear} and \ref{M2_koszul}: set $\cL=\texttt{partialLinearFiltration}(S/I,L)$ and then check that $\texttt{partialKoszulFiltration}(S/I,\cL)$ contains no ideals with two or more minimal generators, hence $S/I$ has no monomial Koszul filtration. 
\end{exam}

\section{A G-quadratic algebra with no Koszul filtration}\label{sec_b4t}
\noindent
The fact that not every Koszul algebra has a Koszul filtration was demonstrated in \cite{CRV2001} by a generic Artinian complete intersection in five variables. Moreover, no such algebra in five or more variables can have a Koszul filtration, and neither is it G-quadratic.
In this section, we give an example which confirms a conjecture of \cite{Ene_2015}: we show there exists a G-quadratic algebra which does not have a Koszul filtration. The example we consider is the Möbius algebra of the cycle matroid of a certain strongly chordal graph.  Such algebras have been recently studied in \cite{LMMP_2025}, where it was shown that the graded Möbius algebra of the cycle matroid of a graph is G-quadratic if and only if it is Koszul---and these conditions are equivalent to the underlying graph being strongly chordal.

\begin{exam}\label{b4t}
Let $\K$ be a field and consider the algebra
\begin{equation}\label{eq_b4t}
\begin{split}
B&=\frac{\K[a,b,c,d,e,f,g,h,i,j,k,l]}{(a^2,b^2,c^2,d^2,e^2,f^2,g^2,h^2,i^2,j^2,k^2,l^2)+\fI} \quad \text{where}
  \\[10pt]
\fI= (&ab-aj,\, ab-bj,\,
ac-ad,\, ac-cd,\,
cg-cj,\,cg-gj,\,
bd-bg,\,bd-dg,\, \\
&de-df,\,de-ef,\,
gh-gi,\,gh-hi,\,
jk-jl,\,jk-kl).
\end{split}
\end{equation}
\noindent
The algebra $B$ arises from the broken 4-trampoline graph in Figure \ref{fig_b4t} by letting each triangle of edges labelled $x, y, z$ give rise to generators $xy-xz$ and $xy-yz$ of the ideal $\fI$.
The algebra $B$ is isomorphic to the graded Möbius algebra of the cycle matroid of the broken 4-trampoline as defined in \cite{LMMP_2025}.
The isomorphism follows from \cite[Proposition 3.1(b)]{LMMP_2025} and the fact that this graph is chordal, along with an induction argument as used in \cite{LMMP_2025}. 
However, the arguments in this paper use solely the presentation \eqref{eq_b4t}.
\begin{figure}[H]
\caption{A labelling of the broken 4-trampoline corresponding to the presentation of the algebra $B$.}\label{fig_b4t}
\begin{center}    
\begin{tikzpicture}[
    vertex/.style={
        circle, 
        draw=black, 
        fill=black, 
        minimum size=3pt, 
        inner sep=0pt
    },
    labelspace/.style={
        inner sep=5pt
    }
]
    \node[vertex] (TL) at (0,1) {}; 
    \node[vertex] (TR) at (1,1) {}; 
    \node[vertex] (BL) at (0,0) {}; 
    \node[vertex] (BR) at (1,0) {}; 
    \node[vertex] (Top)   at (0.5,2) {}; 
    \node[vertex] (Left)  at (-1,0.5) {};
    \node[vertex] (Right) at (2,0.5) {};
    
    \draw (TL) -- (TR) node[midway, above] {$g$};
    \draw (BL) -- (BR) node[midway, below] {$a$};
    \draw (TL) -- (BL) node[midway, left] {$j$};
    \draw (TR) -- (BR) node[midway, right] {$d$};
    \draw (TL) -- (BR) node[pos=0.6, right] {$b$};
    \draw (TR) -- (BL) node[pos=0.6, left] {$c$};
    \draw (Right) -- (BR) node[midway, below] {$e$};
    \draw (Right) -- (TR) node[midway, above] {$f$};
    \draw (Top) -- (TR) node[pos=.6, above, labelspace] {$\: h$};
    \draw (Top) -- (TL) node[pos=.6, above, labelspace] {$i\:$};
    \draw (Left) -- (TL) node[midway, above] {$k$};
    \draw (Left) -- (BL) node[midway, below] {$l$};
\end{tikzpicture}
\end{center}
\end{figure}
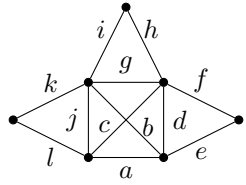
\end{exam}

\begin{rem}\label{rem_b4tGQ}
The algebra $B$ in Example \ref{b4t} is G-quadratic, hence Koszul, independent of the choice of the field $\K$. This follows by \cite[Theorem A]{LMMP_2025}, but can also be seen directly by verifying that with lexicographic ordering such that 
$$a>b>c>e>h>k>l>i>f>j>g>d,$$ one obtains a quadratic Gröbner basis. 
\end{rem}
\begin{rem}\label{rem_b4tKF}
If $\K=\F_2$, the field with 2 elements, then the algebra $B$ in Example \ref{b4t} has a Koszul filtration given by the following set \texttt{LF} of ideals:
\begin{align*}
\{&\{(0)\},\ \{(d),\,(d+e+f)\},\  
\{(e,d),\,(d+e+f,b+e+f+g),\,(g,d+e+f)\},\\
&\{(f,e,d),\,(d+e+f,b+e+f+g,a+c+e+f),\\
&\ (g,d+e+f,b+e+f),\,(g,d+e+f,b+c+f)\},\\
&\{(e+f,d,b+g,a+c),\,(g,d+e+f,c+j,b+e+f),\,(g,e+f,d,b)\},\\
&\{(j,e+f,d,b+g,a+c),\,(e+f,d,c+g+j,b+g,a+g+j),\\
&\ (h+i,g,d+e+f,c+j,b+e+f)\},\\
&\{(j,e+f,d,c,b+g,a),\,(j,e+f,d,c+g,b+g,a+g),\,(g,e+f,d,c+j,b,a+j),\\
&\ (i,h,g,d+e+f,c+j,b+e+f),\,(h+i,g,e+f,d,c+j,b)\},\\
&\{(j,g,e+f,d,c,b,a),\,(k+l,j,e+f,d,c+g,b+g,a+g),\\
&\ (h+i,g,e+f,d,c+j,b,a+j)\},\\
&\{(j,h+i,g,e+f,d,c,b,a),\,(l,k,j,e+f,d,c+g,b+g,a+g)\},\\
&\{(j,i,h,g,e+f,d,c,b,a)\},\ \{(j,i,h,g,f,e,d,c,b,a)\},\\
&\{(k,j,i,h,g,f,e,d,c,b,a)\},\ \{(l,k,j,i,h,g,f,e,d,c,b,a)\}
\}
\end{align*}
\noindent
We verify this in Macaulay2 with the function \texttt{partialKoszulFiltration(B,LF)} defined in Code \ref{M2_koszul}.
\end{rem}
The principal aim of this section is to prove the following.

\begin{thm}\label{thm_b4t}
Let $B$ be the algebra from Example \ref{b4t}. If the base field is $\K=\F_3$, the field with 3 elements, then $B$ has no Koszul filtration. 
\end{thm}

The remainder of this section is devoted to proving this theorem. Henceforth, $B$ denotes the algebra in Example \ref{b4t}, over the base field $\K=\F_3$. Throughout the proof, we utilize Macaulay2 \cite{M2} for computations:
{\footnotesize
\begin{Verbatim}[samepage=true]
i1 : B=ZZ/3[a..l]/ideal(a^2,b^2,c^2,d^2,e^2,f^2,g^2,h^2,i^2,j^2,k^2,l^2,
    a*(b-j),b*(a-j),a*(c-d),c*(a-d),c*(g-j),g*(c-j),b*(d-g),
    d*(b-g),d*(e-f),e*(d-f),g*(h-i),h*(g-i),j*(k-l),k*(j-l));
\end{Verbatim}
}

\begin{lem}\label{lem_37linearforms}
If $w\in B_1$ such that the minimal generators of
$(0):w$ are contained in $B_1$,
then $w$ or $-w$ must be one of the following 37 linear forms: 
$a$, $b$, $c$, $d$, $e$, $f$, $g$, $h$, $i$, $j$, $k$, $l$, 
$a+b$, $a+c$, $d+e$, $d-e$, $d+f$, $d-f$, $e+f$, $g+h$, $g-h$, $g+i$, $g-i$, $h+i$, $j+k$, $j-k$, $j+l$, $j-l$, $k+l$,
$a+b+j$, $a+c+d$, $d+e+f$, $d-e-f$, $g+h+i$, $g-h-i$, $j+k+l$, $j-k-l$.
\end{lem}
\begin{proof}
Define a set \texttt{linearForms} of all linear forms in $B_1$. To remove scalar multiples we remove any linear form whose leading coefficient is not equal to $1$; this also excludes $0$.
{\footnotesize
\begin{Verbatim}[samepage=true]
i2 : L = gens B;
i3 : linearForms = {0_B};
i4 : scan(L, x -> (linearForms = flatten apply(linearForms,v->{v,v+x,v-x});))
i5 : linearForms = select(linearForms,x->leadCoefficient lift(x,ambient B) == 1);
i6 : #linearForms
o6 = 265720
\end{Verbatim}
}
\noindent
The set \texttt{linearForms} now has the expected $(3^{12}-1)/2$ linear forms. Finally, check which of these have a linear annihilator (this takes a while). As $\operatorname{depth}(B)=0$, $(0):w\not=(0)$ for all $w\in B_1$, and so one only needs to look for degree 1 generators:
{\footnotesize
\begin{Verbatim}[samepage=true]
i7 : linearAnn = {};
i8 : for w in linearForms do (
    if member(unique flatten degrees(trim ann(w)),{{1}}) then (
        linearAnn = append(linearAnn,w);
        );
    );
i9 : netList pack(10, linearAnn)
     +-+-+-----+-----+---------+-----+---------+---------+-----+---------+
o9 = |l|k|k + l|j    |j + l    |j - l|j + k    |j + k + l|j - k|j - k - l|
     +-+-+-----+-----+---------+-----+---------+---------+-----+---------+
     |i|h|h + i|g    |g + i    |g - i|g + h    |g + h + i|g - h|g - h - i|
     +-+-+-----+-----+---------+-----+---------+---------+-----+---------+
     |f|e|e + f|d    |d + f    |d - f|d + e    |d + e + f|d - e|d - e - f|
     +-+-+-----+-----+---------+-----+---------+---------+-----+---------+
     |c|b|a    |a + c|a + c + d|a + b|a + b + j|         |     |         |
     +-+-+-----+-----+---------+-----+---------+---------+-----+---------+
\end{Verbatim}
}
\noindent
This is the desired set.
\end{proof}

Let us say that an ideal of $B$ is \emph{forbidden} if it can not appear in any Koszul filtration of $B$. For instance, an ideal is forbidden if its minimal free resolution is not linear \cite[Proposition 1.2]{CTV2001}.  
Our strategy to prove Theorem \ref{thm_b4t} is to show that 
every principal ideal generated by one of the linear forms in Lemma \ref{lem_37linearforms} is forbidden. 

\begin{lem}\label{lem_abc}
The ideals $(a)$, $(b)$, $(c)$, $(a+b)$, $(a+c)$, $(a+b+j)$, and $(a+c+d)$ do not belong to a Koszul filtration of $B$.
\end{lem}
\begin{proof}
For $w\in B_1$, a necessary condition for the principal ideal $(w)$ to belong to a Koszul filtration is that its annihilator $(0):w$ belongs to the Koszul filtration as well. Thus it suffices to check that the annihilator of each ideal has a non-linear minimal free resolution. For example, the following computation shows that the ideal $(0):a$ has a non-linear minimal free resolution; it has a cubic minimal generator in degree 3:

{\footnotesize
\begin{Verbatim}[samepage=true]
i10 : betti res(ann(a),LengthLimit=>3)
             0 1  2  3
o10 = total: 1 3 12 56
          0: 1 3 12 55
          1: . .  .  .
          2: . .  .  1
o10 : BettiTally
\end{Verbatim}
}
\noindent
The other linear forms are checked in the same way, and in those cases one only needs to check the first 2 steps of the resolution.
\end{proof}

Another way to show that an ideal $\fa \subset B$ is forbidden is to consider a \emph{partial linear flag of $\fa$},  meaning a chain of ideals
$$\cF \colonequals \{(0)=\fa_0\subset \fa_1\subset \cdots \subset \fa_t=\fa\}$$ 
such that the ideals $\fa_s$ and $\fa_s:\fa_{s+1}$ are minimally generated by linear elements, and $\fa_s$ has $s$ minimal generators. If for every partial linear flag $\cF$ of $\fa$ the set $\cC_{\cF}\colonequals \{(\fa_{s-1}:\fa_s)\}_{s=1}^t$ contains at least one forbidden ideal, then $\fa$ itself cannot belong to a Koszul filtration of $B$. This follows directly from the definition. The following Macaulay2 function will help with this verification process.

\begin{code}\label{M2_colons}
We adapt Algorithm \ref{alg_findflags} to create a Macaulay2 function suited for this section. The inputs of the function include a standard graded algebra \texttt{R} over $\F_3$ and two sets of linear forms: \texttt{L} is a chosen list of minimal generators of an ideal $\fa$, and \texttt{P} is a list of some linear forms in $\fa$. The function returns a list of the unique lists of colon ideals (excluding the first annihilator ideal) that appear in all possible partial linear flags of $\fa$, and that include a principal ideal generated by an element of \texttt{P}.
{\footnotesize
\begin{Verbatim}[samepage=true]
i11 : linearColonIdealsZ3 = (R,L,P) -> (
    colonSets = {}; n = #L; Lall = {0_R};
    scan(L, x -> (Lall = flatten apply(Lall, v -> {v, v + x, v - x});));
    Lall = select(Lall,x->leadCoefficient lift(x,ambient R) == 1);
    for startGen in P do (
        for idealGens in flatten apply(subsets(Lall,n-1), T -> permutations T) do (
            colons = {};
            if ideal(join({startGen},idealGens))==ideal(L) then (
                for i from 0 to n-2 do (
                    CI = trim (ideal(join({startGen},idealGens_{0..(i-1)})):idealGens_i);
                    colons = append(colons,CI);
                    );
                if all(colons,c->member(unique flatten degrees(c),{{1}})) then (
                    colonSets = unique append(colonSets,colons);
                    );
                );
            );
        );
    return colonSets
    );
\end{Verbatim}
}
\end{code}

To reduce which ideals from Lemma \ref{lem_37linearforms} that we need to check, we take symmetry of the algebra into account. The permutations in the next lemma are presented in cycle notation, so for example $(ef)$ denotes the transposition of $e$ and $f$.
\begin{lem}\label{lem_permutation}
    The permutations of the variables $\{a,b, \ldots, l\}$
    \[
    \pi_1=(ef), \quad \pi_2=(hi), \quad \pi_3=(kl), \quad \pi_4=(bc)(dj)(el)(fk)(hi)
    \]
    induce degree preserving automorphisms on the algebra $B$.
\end{lem}
\begin{proof}
    Let us first consider the permutation $\pi_1$. Notice that 
    \[df-ef=(de-ef)-(de-df) \in \fI.\]
    We see that 
    \[
    \pi_1(e^2)=f^2, \quad \pi_1(f^2)=e^2, \quad \text{and}
    \]
    \[
    \pi_1(de-df)=-(de-df), \quad \pi_1(de-ef)=df-ef, \quad \pi_1(df-ef)=de-ef,
    \]
    and all other generators of the defining ideal are fixed under $\pi_1$. It follows that 
    \[
    \pi_1((a^2, \ldots, l^2) + \fI)=(a^2, \ldots, l^2) + \fI.
    \]
    Similar arguments apply to $\pi_2$ and $\pi_3$. 

    For $\pi_4$ notice that the graph in Figure \ref{fig_b4t} is invariant under the permutation $\pi_4$ applied to its edges. Hence the defining ideal of the corresponding Möbius algebra will be identical. Alternatively, it is straightforward to check that the defining ideal is indeed invariant under the action by $\pi_4$. 
\end{proof}

\begin{lem}\label{lem_dgj}
The ideals $(d)$, $(g)$, and $(j)$ do not belong to a Koszul filtration of $B$.
\end{lem}
\begin{proof}
First, we argue the case for the ideal $(d)$. Start by computing $(0):d$.
{\footnotesize
\begin{Verbatim}[samepage=true]
i12 : ann(d)
o12 = ideal (e - f, d, b - g, a - c)
\end{Verbatim}
}
\noindent
The only linear form in $(0):d$ which appears in the list given by Lemma \ref{lem_37linearforms} is $d$; all other linear forms in $(0):d$ generate forbidden principal ideals. Thus we may use Code \ref{M2_colons} to compute the following (unnecessary output has been truncated):
{\footnotesize
\begin{Verbatim}[samepage=true]
i13 : netList linearColonIdealsZ3(B,{e-f,d,b-g,a-c},{d})
      +-------------------------------+--------------------------------------+---+
o13 = |ideal (d, c, a)                |ideal (j, g, d, c, b, a)              |...|
      +-------------------------------+--------------------------------------+---+
      |ideal (d, c, a)                |ideal (f, e, d, a - c)                |...|
      +-------------------------------+--------------------------------------+---+
      |ideal (f, e, d)                |ideal (e - f, d, c, a)                |...|
      +-------------------------------+--------------------------------------+---+
      |ideal (d, b + c + j, a + g + j)|ideal (j, g, d, c, b, a)              |...|
      +-------------------------------+--------------------------------------+---+
      |ideal (d, b + c + j, a + g + j)|ideal (f, e, d, a + b - c - g)        |...|
      +-------------------------------+--------------------------------------+---+
      |ideal (f, e, d)                |ideal (e - f, d, b + c + j, a + g + j)|...|
      +-------------------------------+--------------------------------------+---+
\end{Verbatim}
}
\noindent 
This yields six distinct sets of colon ideals, each of which contains one of the ideals
$$(d,c,a), (e-f,d,c,a), (d,b+c+j,a+g+j), (e-f,d,b+c+j,a+g+j).$$ 
In other words, if $(0):d$ belongs to a Koszul filtration of $B$, then at least one of these ideals must also belong to the Koszul filtration. However, these four ideals have non-linear minimal free resolutions. Check each of these as before; for example
{\footnotesize
\begin{Verbatim}[samepage=true]
i14 : betti res(ideal(d,c,a),LengthLimit=>3)
             0 1  2  3
o14 = total: 1 3 12 55
          0: 1 3 12 54
          1: . .  .  1
o14 : BettiTally
\end{Verbatim}
}
\noindent
shows that $(d,c,a)$ is forbidden. The others are checked similarly. Thus $(0):d$ is forbidden, and hence so is $(d)$.

The permutation $\pi_4$ from Lemma \ref{lem_permutation} shows that $(j)$ is forbidden. The case for $(g)$ is similar to that of $(d)$, so we only sketch the idea: First notice that $(0):g=(h-i,g,c-j,b-d)$, and that (again using \texttt{linearColonIdealsZ3}) inclusion of this ideal in a Koszul filtration would force the inclusion of at least one of the ideals $(g, b - c - d + j, a)$, $(h - i, g, b - c - d + j, a)$, $(g, b + c - d - j, a + d + j)$, or $(h - i, g, b + c - d - j, a + d + j)$. All of these have non-linear minimal free resolutions, even within 3 steps, and hence $(g)$ is forbidden as well.
\end{proof}

\begin{lem}\label{lem_fed}
The ideals $(f,e,d)$, $(i,h,g)$, and $(l,k,j)$ do not belong to a Koszul filtration of $B$.
\end{lem}
\begin{proof}
We first show that the ideal $(f,e,d)$ is forbidden.

The set $P=\{f,e,e+f,d+f,d-f,d+e,d+e+f,d-e,d-e-f\}$ consists of all linear forms in $(f,e,d)$ which appear in the list given by Lemma \ref{lem_37linearforms},
excluding $d$, since we know by Lemma \ref{lem_dgj} that $(d)$ is forbidden. Now use Code \ref{M2_colons} to compute:
{\footnotesize
\begin{Verbatim}[samepage=true]
i15 : netList linearColonIdealsZ3(B,{f,e,d},{f,e,e+f,d+f,d-f,d+e,d+e+f,d-e,d-e-f})
      +-----------------------------+-----------------------------+
o15 = |ideal (f, e, d, b - g, a - c)|ideal (f, e, d)              |
      +-----------------------------+-----------------------------+
      |ideal (f, e, d)              |ideal (f, e, d, b - g, a - c)|
      +-----------------------------+-----------------------------+
\end{Verbatim}
}
\noindent
This output shows that every set of colon ideals (of a partial linear flag of $(f,e,d)$ passing through a principal ideal generated by an element of $P$) includes the ideal $\fa=(f,e,d,b-g,a-c)$. Thus, if $(f,e,d)$ belongs to a Koszul filtration of $B$, then so does $\fa$. We next argue that this is impossible, that is, $\fa$ is a forbidden ideal.

To do this, we analyze partial linear flags, and their colon ideals, of $\fa$. This turns out to be computationally challenging to do directly as before with the function \texttt{linearColonIdealsZ3}, and so instead we break up the problem into cases.
Consider the element $b-g\in \fa$. A partial linear flag of $\fa$ has the form 
$$(0)=\fa_0\subset \fa_1\subset \fa_2\subset \fa_3\subset \fa_4 \subset \fa_5=\fa.$$
The principal ideal $\fa_1$ is generated by one of the linear forms in $P$, as those are the only linear forms contained in $\fa$ allowed by Lemma \ref{lem_37linearforms}.
In particular, $\fa_1\not=(b-g)$ so 
it must be the case that $b-g\in \fa_s\setminus \fa_{s-1}$ for some $s=2,\ldots,5$.
This gives 4 cases to consider. We claim that for $s=2,3$, the colon ideal $\fa_{s-1}:\fa_s$ has quadratic minimal generators, and that for $s=4,5$, the ideal $\fa_{s-1}$ is forbidden. First, set 
{\footnotesize
\begin{Verbatim}[samepage=true]
i16 : L = {f,e,d,b-g,a-c};
i17 : P = {f,e,e+f,d+f,d-f,d+e,d+e+f,d-e,d-e-f};
\end{Verbatim}
}

\noindent
Case 1: $b-g\in \fa_2\setminus \fa_1$. 

For every $w$ in \texttt{P}, the colon ideal $(w):b-g$ contains quadratic minimal generators:
{\footnotesize
\begin{Verbatim}[samepage=true]
i18 : colonDegs = {};
i19 : for startGen in P do (
          colon = trim (ideal(startGen):b-g);
          colonDegs = unique append(colonDegs,unique degrees(colon));
          );
i20 : colonDegs 
o20 = {{{1}, {2}}}
o20 : List
\end{Verbatim}
}

\noindent
Case 2: $b-g\in \fa_3\setminus \fa_2$.

Again, check that all colon ideals $\fa_2:b-g$ contain quadratic minimal generators:
{\footnotesize
\begin{Verbatim}[samepage=true]
i21 : Lall = {0_B};
i22 : scan(L, x -> (Lall = flatten apply(Lall, v -> {v, v + x, v - x});))
i23 : Lall = select(Lall,x->leadCoefficient lift(x,ambient B) == 1);
i24 : colonDegs = {};
i25 : for startGen in P do (
          for idealGens in subsets(Lall,1) do (
              if not isSubset(ideal(b-g),ideal(startGen,idealGens_0)) then (
                  CI = trim (ideal(startGen,idealGens_0):b-g);
                  colonDegs = unique append(colonDegs,unique degrees(CI));
                  );
              );
          );
i26 : colonDegs
o26 = {{{1}, {2}}}
o26 : List
\end{Verbatim}
}

\noindent
Case 3: $b-g\in \fa_4\setminus \fa_3$.

We show that every choice of $\fa_3$ in a partial linear flag of $\fa$ is forbidden. Start by computing a set \texttt{subIdeals} of all possible ideals $\fa_3$: 
{\footnotesize
\begin{Verbatim}[samepage=true]
i27 : subIdeals = {};
i28 : for startGen in P do (
          for idealGens in subsets(Lall,2) do (
              if not isSubset(ideal(b-g),ideal(startGen,idealGens_0,idealGens_1)) 
                  then (
                      CI = trim (ideal(startGen,idealGens_0,idealGens_1):b-g);
                      if member(unique flatten degrees(CI),{{1}}) then (
                          SI = ideal(startGen,idealGens_0,idealGens_1);
                          subIdeals = unique append(subIdeals,trim SI);
                          );
                      );
              );
          );
i29 : #subIdeals
o29 : 630
\end{Verbatim}
}
\noindent
Next, check that the 630 ideals in \texttt{subIdeals} have non-linear minimal free resolutions. Here is a small helper function to check linearity for the first few steps:
{\footnotesize
\begin{Verbatim}[samepage=true]
i30 : resLooksLinear = (J,t) -> (
          F = res(J,LengthLimit=>t); K = keys betti F;
          rowIndices = unique apply(K,k->(k_1)_0-k_0);
          #rowIndices == 1
          )
o30 = resLooksLinear
o30 : FunctionClosure
\end{Verbatim}
}
\noindent
This returns \texttt{true} if the first $t$ steps of the minimal free resolution are linear. Use this to check whether these 630 ideals are linear for the first five steps:
{\footnotesize
\begin{Verbatim}[samepage=true]
i31 : for I in subIdeals do (
          if resLooksLinear(I,4)==true then (
              if resLooksLinear(I,5)==true then (
                  print I;
                  );
              );
          );
\end{Verbatim}
}
This returns no ideals; in other words, all 630 ideals have non-linear minimal free resolutions, and are thus forbidden.

\noindent 
Case 4: $b-g\in \fa\setminus \fa_4$.

In this case, elementary linear algebra yields
$$\fa_4=(f+\lambda_1(b-g),e+\lambda_2(b-g),d+\lambda_3(b-g),(a-c)+\lambda_4(b-g)),$$ for coefficients $\lambda_i\in \{0,1,-1\}$. Now, we compute the minimal free resolution for all possible $\fa_4$ for the first four steps:
{\footnotesize
\begin{Verbatim}[samepage=true]
i32 : Lbase = {f,e,d,a-c};
i33 : Ldiffs = {0_B,b-g,-b+g};
i34 : L = apply(Lbase, v -> apply(Ldiffs, d -> v + d));
i35 : for w0 in L_0 do for w1 in L_1 do for w2 in L_2 do for w3 in L_3 do (
          if resLooksLinear(ideal(w0,w1,w2,w3),4) then (
              print ideal(w0,w1,w2,w3);
              );
          );
\end{Verbatim}
}
\noindent
This prints no ideals, and thus all have non-linear minimal free resolutions. Thus $\fa$ is a forbidden ideal, and hence $(f,e,d)$ is forbidden as well. 

Because $(f,e,d)$ is forbidden, the permutation $\pi_4$ from Lemma \ref{lem_permutation} shows that $(l,k,j)$ is also forbidden. The argument for showing $(i,h,g)$ is forbidden is similar to that of $(f,e,d)$, so we only sketch the argument. First, observe that the inclusion of any possible partial linear flag of $(i,h,g)$ forces the inclusion of the ideal $\fb=(i,h,g,c-j,b-d)$, here using also that such a flag cannot contain the principal ideal $(g)$ by Lemma \ref{lem_dgj}. One then argues that $\fb$ is forbidden by considering all partial linear flags $(0)=\fb_0\subset \fb_1\subset \cdots \subset \fb_5=\fb$ and whether $c-j\in \fb_s$ for each $s$. Since $(c-j)$ is forbidden by Lemma \ref{lem_37linearforms}, there are four cases to consider: $c-j\in \fb_s\setminus \fb_{s-1}$ for $s=2,\ldots,5$. For $s=2,3$ argue that if $c-j\in \fb_s\setminus \fb_{s-1}$ then $\fb_{s-1}:\fb_s$ has quadratic minimal generators. For $s=4,5$, every choice of the ideal $\fb_{s-1}$ has a non-linear minimal free resolution.
\end{proof}

\begin{proof}[Proof of Theorem \ref{thm_b4t}]
By definition, any Koszul filtration must have a principal ideal. Thus, if a Koszul filtration of $B$ exists,  it must contain some principal ideal generated by one of the 37 linear forms given in Lemma \ref{lem_37linearforms}. In light of Lemma \ref{lem_abc}, the principal generator cannot have any term whose support belongs to $\{a,b,c\}$. Further, by Lemma \ref{lem_dgj} the ideals $(d)$, $(g)$, and $(j)$ are forbidden from a Koszul filtration. This leaves 27 linear forms. Further accounting for the permutations of variables from Lemma \ref{lem_permutation} leaves only 14 linear forms: $e$, $h$, $d+e$, $d-e$, $e+f$, $g+h$, $g-h$, $h+i$, $j+k$, $j-k$, $d+e+f$, $d-e-f$, $g+h+i$, $g-h-i$.

Suppose that $(e)$ belongs to a Koszul filtration of $B$. Then:
{\footnotesize
\begin{Verbatim}[samepage=true]
i36 : ann(e)
o36 = ideal (e, d - f)
o36 : Ideal of B
i37 : linearColonIdealsZ3(B,{e,d-f},{e,d-f,e-d+f})
o37 = {{ideal (f, e, d)}}
o37 : List
\end{Verbatim}
}
\noindent
Thus, as $(f,e,d)$ is forbidden, so is $(e)$. The other 13 ideals proceed in exactly the same way: the inclusion of $(d+e)$, $(d-e)$, $(e+f)$, $(d+e+f)$, or $(d-e-f)$ would force the forbidden ideal $(f,e,d)$ to be included, the inclusion of $(j+k)$ or $(j-k)$ would force the forbidden ideal $(l,k,j)$ to be included, and the inclusion of $(h)$, $(g+h)$, $(g-h)$, $(h+i)$, $(g+h+i)$, or $(g-h-i)$ would force the forbidden ideal $(i,h,g)$ to be included. Thus no principal ideal can belong to a Koszul filtration of $B$, and so $B$ has no Koszul filtration.
\end{proof}

\begin{rem}
It was shown in \cite{LMMP_2025} that the graded Möbius algebra of the cycle matroid of the broken $n$-trampoline graph is G-quadratic for $n\geq 3$. We claim that $n=4$ is the smallest case where this algebra can fail to have a Koszul filtration. By Theorem \ref{thm_b4t}, it suffices to show that for $n=3$ this algebra has a Koszul filtration over any field $\K$. The graded Möbius algebra of the cycle matroid of the broken 3-trampoline
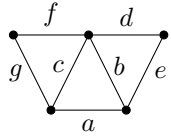
\begin{center}
\begin{tikzpicture}[
    vertex/.style={
        circle, 
        draw=black, 
        fill=black, 
        minimum size=3pt, 
        inner sep=0pt
    },
    labelspace/.style={
        inner sep=5pt
    }
]
    \node[vertex] (TL) at (0,1) {}; 
    \node[vertex] (TR) at (2,1) {}; 
    \node[vertex] (BL) at (0.5,0) {}; 
    \node[vertex] (BR) at (1.5,0) {}; 
    \node[vertex] (Top)   at (1,1) {};

    \draw (Top) -- (BR) node[pos=0.4, right] {$b$};
    \draw (BL) -- (BR) node[midway, below] {$a$};
    \draw (Top) -- (BL) node[pos=0.4, left] {$c$};
    \draw (TR) -- (BR) node[midway, right] {$e$};
    \draw (TL) -- (BL) node[midway, left] {$g$};
    \draw (TR) -- (Top) node[midway, above] {$d$};
    \draw (Top) -- (TL) node[midway, above] {$f$};
\end{tikzpicture}
\end{center}
\noindent
is presented as
$$B'=\frac{\K[a,b,c,d,e,f,g]}{(a^2,b^2,c^2,d^2,e^2,f^2,g^2,gc-fc,gf-fc,db-eb,de-eb,ac-bc,ab-bc)}.$$
The given generators of the defining ideal of $B'$ form a quadratic Gröbner basis with the lexicographic order, such that $a>d>g>f>e>b>c$. In fact, we argue that 
$$\{(0),(e),(e,d-b),(e,b,d),(e,b,d,a-c),(e,b,d,a,c),(e,b,d,a,c,g-f),\m_{B'}\}$$
 is a Gröbner flag of $B'$, independent of the characteristic.
Let $G_0$ be the given set of generators for the defining ideal of $B'$, and set $\alpha_1=e$, $\alpha_2=d-b$, $\alpha_3=b$, $\alpha_4=a-c$, $\alpha_5=c$, $\alpha_6=g-f$, and $\alpha_7=f$. Further define $G_s=G_0\cup\{\alpha_1,\ldots,\alpha_s\}$. Each $G_s$ is a Gröbner basis, independent of the characteristic of the underlying field. We verify in Macaulay2 that if $\operatorname{char}(\K)=0$, then
\begin{equation}\label{eq_broken3trampGB}
(G_s):\alpha_{s+1}=(G_{s+2}) \quad \text{for } s=0,\ldots ,5, \quad \text{and } (G_6):\alpha_7=(G_7).
\end{equation}
\noindent
Proceed with an argument similar to that of \cite[Proof of Lemma 3.6]{DAli2017}. A direct computation shows that the containments $(G_s):\alpha_{s+1}\supseteq (G_{s+2})$ and $(G_6):\alpha_7\supseteq (G_7)$ hold independent of characteristic. Moreover, for each $s=0, \ldots,6$, multiplication by $\alpha_{s+1}$ induces a short exact sequence of graded $S$-modules:
\[\xymatrix{
0 \ar[r] & S/((G_s):\alpha_{s+1})[-1] \ar[r] & S/(G_s) \ar[r] & S/(G_{s+1}) \ar[r] & 0
}
\]
The Hilbert series of the middle and right modules in this sequence are independent of the characteristic. This follows because the ideals $G_s$ and $G_{s+1}$ are Gröbner bases for $(G_s)$ and $(G_{s+1})$, respectively. 
Hilbert series are additive on short exact sequences, so these two series determine the series for $S/((G_s):\alpha_{s+1})$, and hence this Hilbert series is also independent of the characteristic. Checking it is the expected series in characteristic $0$ is thus sufficient to conclude the equalities in \eqref{eq_broken3trampGB} hold in any characteristic.
\end{rem}

We end with a question: does the algebra in Example \ref{b4t} have a Koszul filtration over any field of characteristic different from $2$? Evidently the filtration in Remark \ref{rem_b4tKF} is still valid over any field of characteristic $2$. On the other hand, we expect that no Koszul filtration exists when $\operatorname{char}\K\neq 2$, but at present can verify this only for $\K=\F_3$, by Theorem \ref{thm_b4t}.

\section*{Acknowledgements}
We thank the two anonymous referees for their careful reading and valuable comments on a draft of this manuscript.

\appendix
\section{Algorithms in Macaulay2}\label{sec_appendix}
We implement in Macaulay2 (version 1.21) \cite{M2} the algorithms from Section \ref{sec_alg}. Let $R$ be a standard graded $\K$-algebra and $L\subseteq R_1$ a list of linear elements of $R$.

\begin{code}\label{M2_linear}
The following is an implementation of Algorithm \ref{alg_partial_linear}. Application of the function $\texttt{partialLinearFiltration}(R,L)$ produces a partial linear filtration of $R$, which consists of ideals and colon ideals that are generated by elements of $L$.

{\footnotesize
\begin{Verbatim}[samepage=true]
partialLinearFiltration = (R,L) -> (
    linearIdeals = unique apply(subsets L, a->trim promote(ideal(a),R));
    F = {{promote(ideal(),R)}};
    for i from 1 to numgens R do (
        newIdeals = {};
        for I in F#(i-1) do ( 
            for a in L do (
                if isSubset(ideal(a), I) then continue;
                colIdeal = trim I:a;
                if member(colIdeal,linearIdeals) then (
                    toInclude = trim(I+ideal(a));
                    newIdeals = unique append(newIdeals,toInclude);
                    );
                );
            );
        F = append(F,newIdeals);
        );
    return F;
    );
\end{Verbatim}
}
\end{code}

\begin{code}\label{M2_koszul}
Next, we give an implementation of Algorithm \ref{alg_partial_koszul}. Let $LF$ be a linear filtration of $R$, presented as a list of lists of ideals in $R$. Application of $\texttt{partialKoszulFiltration}(R,LF)$ produces a partial Koszul filtration of $R$.

{\footnotesize
\begin{Verbatim}[samepage=true]
partialKoszulFiltration = (R,LF) -> (
    F = {}; Fprev = {};
    for i from 0 to #LF-1 do (
        F = append(F,unique apply(LF#(i),I->trim I));
        );
    while F =!= Fprev do (
        Ftrim = {{promote(ideal(),R)}};
        for i from 1 to #F-1 do (
            FtrimLevel = {};
            for I1 in F#(i-1) do (
                for I2 in F#(i) do (   
                    if isSubset(I1,I2) then (
                        colIdeal = trim I1 : I2; n=numgens colIdeal;
                        if colIdeal == ideal(0_R) or member(colIdeal,F#(n)) then (
                            FtrimLevel = append(FtrimLevel, I2);
                            );
                        );
                    );
                );
            FtrimLevel = unique FtrimLevel;
            Ftrim = append(Ftrim,FtrimLevel);
            );
        Fprev = F; F = Ftrim;
        );
    return unique F;
    );
\end{Verbatim}
}
\noindent
Recall that if the output contains the maximal ideal of $R$, then it is a Koszul filtration. Often, we take $LF=\texttt{partialLinearFiltration}(R,L)$ from \ref{M2_linear}. On the other hand, taken independently this function can be used to verify a given set of linear ideals is a Koszul filtration, and if it is not, then trim it down to the largest Koszul filtration contained in it (if it exists). In this way, it can be seen as an improvement of the function \texttt{IsKoszulFiltration} written in CoCoA in \cite{BN2013}.
\end{code}

\begin{code}\label{M2_flags}
Finally, we implement Algorithm \ref{alg_findflags} in Macaulay2. Let $\fa$ be an ideal of $R$ with $s$ minimal linear generators. The function $\texttt{partialLinearFlags}(R,L,\fa)$ produces a list of all partial linear flags of $\fa$ having generators from $L$. It returns the empty set if there are no such flags. \\
{\footnotesize
\begin{Verbatim}[samepage=true]
partialLinearFlags = (R,L,I)  -> (
    linearIdeals = unique apply(subsets L, a->trim promote(ideal(a),R));
    Flags = {{I}}; i = numgens trim I;
    while  i > 0 and #Flags =!= 0 do (
        Flagsnew = {}; i = i - 1;    
        for F in Flags do (
            for a in linearIdeals do (
                if numgens trim a == i and isSubset(a, F#-1) then (
                    colIdeal = a:(F#-1);	    
                    if isSubset(unique degrees trim colIdeal,{{},{1}}) then (
                        Fnew = append(F,a);
                        Flagsnew = append(Flagsnew,Fnew);
                        );
                    );
                );
            );
        Flags = Flagsnew;
        );
    return Flags
    );
\end{Verbatim}
}
\end{code}


\providecommand{\bysame}{\leavevmode\hbox to3em{\hrulefill}\thinspace}
\providecommand{\MR}{\relax\ifhmode\unskip\space\fi MR }
\providecommand{\MRhref}[2]{%
  \href{http://www.ams.org/mathscinet-getitem?mr=#1}{#2}
}
\providecommand{\href}[2]{#2}

\end{document}